\documentclass[11pt,oneside,english,jou]{amsart}
\usepackage[T1]{fontenc}
\usepackage[latin9]{inputenc}
\usepackage{verbatim}
\usepackage{amsthm}
\usepackage{amstext}
\usepackage{amssymb}
\usepackage{url}
\usepackage{hyperref}
\makeatletter
\numberwithin{equation}{section}
\numberwithin{figure}{section}
\theoremstyle{plain}
\newtheorem{thm}{Theorem}[section]
  \theoremstyle{definition}
  \newtheorem{rem}[thm]{Remark}
  \theoremstyle{definition}
  \newtheorem{defn}[thm]{Definition}

  \theoremstyle{definition}
  \newtheorem{subrem}{Remark}[subsection]
  \theoremstyle{definition}
  \newtheorem{subdefn}[subrem]{Definition}

  \theoremstyle{plain}
  \newtheorem{prop}[thm]{Proposition}
  \theoremstyle{definition}
  \newtheorem{problem}[thm]{Problem}
  \theoremstyle{plain}
  \newtheorem{lem}[thm]{Lemma}

\newcommand{\R}{\mathbb{R}}

\newcommand{\Z}{\mathbb{Z}}

\makeatother

\usepackage{babel}

\makeatother

\usepackage{babel}

\newcommand{\shift}{\operatorname{shift}}
\newcommand{\diam}{\operatorname{diam}}
\newcommand{\ord}{\operatorname{ord}}
\newcommand{\mdim}{\operatorname{mdim}}
\newcommand{\edim}{\operatorname{edim}}
\newcommand{\perdim}{\operatorname{perdim}}
\newcommand{\orb}{\operatorname{orb}}
\newcommand{\co}{\operatorname{co}}
\newcommand{\spann}{\operatorname{span}}

\begin{document}

\title{Mean Dimension \& Jaworski-type Theorems}

\author{Yonatan Gutman}
\begin{abstract}
According to the celebrated Jaworski Theorem, a finite dimensional
aperiodic dynamical system $(X,T)$ embeds in the $1$-dimensional
cubical shift $([0,1]{}^{\mathbb{Z}},\shift)$. If $X$ admits periodic
points (still assuming $\dim(X)<\infty$) then we show in this paper that periodic
dimension $\perdim(X,T)<\frac{d}{2}$ implies that $(X,T)$ embeds
in the $d$-dimensional cubical shift $(([0,1]^{d})^{\mathbb{Z}},\shift)$. This verifies a conjecture by Lindenstrauss and Tsukamoto for finite dimensional systems.
Moreover for an infinite dimensional dynamical system, with the same
periodic dimension assumption, the set of periodic points can be equivariantly
immersed in $(([0,1]^{d})^{\mathbb{Z}},\shift)$. Furthermore we introduce a notion of markers for general topological dynamical
systems, and use a generalized version of the Bonatti-Crovisier tower
theorem, to show that an extension $(X,T)$ of an aperiodic finite-dimensional
system whose mean dimension obeys $\mdim(X,T)<\frac{d}{16}$ embeds
in the \textit{$(d+1)$-}cubical shift $(([0,1]^{d+1})^{\mathbb{Z}},\shift)$.

\date{\today}
\end{abstract}

\keywords{Mean dimension, periodic dimension, Jaworski's theorem, cubical shift,
marker property. }

\subjclass[2010]{37A35, 54F45.}

\maketitle

\section{Introduction\label{sec:Introduction}}

The classical Theorem of Jaworski (1974) is usually stated in the
following way: a finite-dimensional aperiodic (invertible) topological
dynamical system can be equivariantly embedded in the \textit{$1$-dimensional
cubical shift} $([0,1]^{\mathbb{Z}},\shift)$. However a careful reading
of Jaworski's PhD thesis reveals that Jaworski actually proved a
slightly stronger statement: one can replace the assumption of no
periodic points with the requirement of no periodic points of order
less or equal to $(4n+3)^{2}$ (where $n$ is the dimension of the
system). One may wonder if this constant is best possible. Indeed
it is pointed out by Coornaert (2005), that using Jaworski's original
argument, it is enough to assume there are no periodic points of order
less or equal to $6n$. One may try to find the optimal constant (it
turns out to be $2n$) but it soon becomes clear this is not the right
way to recast Jaworski's theorem. Instead, it is better to consider
the obstruction to embedding due to the set of periodic points in
its totality. Our first theorem is the statement that if $X$ is finite-dimensional
and $\perdim(X,T)<\frac{d}{2}$ (see definition in Subsection \ref{subsec:perdim}), then $(X,T)$ embeds in $(([0,1]^{d})^{\mathbb{Z}},\shift)$.
It is an old result due to Flores (1935) that one can find metric
compacta $X_{d}$ of dimension $d$ which do not embed in $\mathbb{R}^{2d}$.
Considering the trivial action $Id:X_{d}\rightarrow X_{d}$, we see
there exists systems with $\perdim(X_{d},Id)=d$ which cannot be embedded
in $(([0,1]^{2d})^{\mathbb{Z}},\shift)$. The constant in the theorem
is thus optimal. The new theorem is a particular case of a conjecture
by Lindenstrauss and Tsukamoto (2012): If $\mdim(X,T)<\frac{d}{2}$
and $\perdim(X,T)<\frac{d}{2}$ then $(X,T)$ embeds in $(([0,1]^{d})^{\mathbb{Z}},\shift)$.
While we cannot report any progress in establishing this theorem in
its totality, we have been able to prove what can be dubbed "the
easy half'' of the conjecture: namely if $\perdim(X,T)<\frac{d}{2}$
then one can find a morphism $\phi:(X,T)\rightarrow(([0,1]^{d})^{\mathbb{Z}},\shift)$
so that for $P$, the set of all periodic points,
$\phi_{|P}:(P,T)\hookrightarrow(([0,1]^{d})^{\mathbb{Z}},\shift)$
is an immersion, i.e. an injective continuous mapping.

Another direction for which Jaworski's theorem lends itself for generalization
is replacing the requirement of the phase space $X$ being finite
dimensional by the condition of possessing a finite dimensional factor.
This returns to Lindenstrauss' ingenious theorem (1999), stating that
an extension of an aperiodic minimal system with $\mdim(X,T)<\frac{d}{36}$
embeds in $(([0,1]^{d})^{\mathbb{Z}},\shift)$. We show that an extension
of an aperiodic finite dimensional system with $\mdim(X,T)<\frac{d}{16}$
embeds in $(([0,1]^{d+1})^{\mathbb{Z}},\shift)$. In order to prove
this theorem we introduce for general dynamical systems the notion of a \textit{marker},
well known in the setting of zero dimensional systems, since the fundamental
work of Krieger (1982). The key step is the statement the aperiodic
finite dimensional systems admit markers. Surprisingly this is achieved
by a certain generalization of the tower theorem of Bonatti-Crovisier
(2004).

\subsection*{Remark}
W. Winter and G. Szab\'{o} have informed me that the theorem that aperiodic
finite dimensional systems admit markers have found applications in the classification theory for C*-algebras arising from dynamical systems (see \cite{S15,HWZ12}).

\subsection*{Acknowledgements}

I would like to thank Jérôme Buzzi, Sylvain Crovisier, Tomasz Downarowicz,
Misha Gromov, Mariusz Lema\'{n}czyk, Elon Lindenstrauss, Julien Melleray,
Micha\l{}  Rams and Benjamin Weiss for helpful discussions. Special
thanks to Tomasz Downarowicz and the Wroc\l{}aw University of Technology
for hosting me for numerous times during the time I was working on
this article. I am grateful to Hanfeng Li and Masaki Tsukamoto for sending me comments on an earlier version. I thank the anonymous referee for a careful reading and many useful suggestions. The research was partially supported by the Marie Curie grant PCIG12-GA-2012-334564 and by the National Science Center (Poland)
grant 2013/08/A/ST1/00275.

\section{Preliminaries\label{sec:Preliminaries}}

We expand on some of the notions appearing in the Introduction:

\subsection{Embedding Dimension}

Let $X$ be a metric compact space and $T : ~X ~\rightarrow ~ X$ a homeomorphism.
For the topological dynamical system (t.d.s) $(X,T)$, define:

\[
\edim(X,T)=\min\{d\in\mathbb{N}\cup{\infty}|(X,T)\hookrightarrow(([0,1]^{d})^{\mathbb{Z}},\shift)
\]
This is the minimal $d$ such that there is a continuous equivariant
embedding of $(X,T)$ into the shift on the $d$-cube, the $d$-\textit{dimensional
cubical shift. }Notice that any compact metric space can be embedded
in $[0,1]^{\mathbb{N}}$ which implies $(X,T)$ is naturally embedded
in $(([0,1]^{\mathbb{N}})^{\mathbb{Z}},\shift)$, hence $\edim(X,T)$
is well defined.

\subsection{Topological Generators}

The embedding dimension can be understood as a topological dynamical
analogue to some concepts used in measured dynamics. Let $(X,\mathcal{B},\mu,T)$
be an (invertible) measure preserving system. A measurable function
$P:X\rightarrow\{1,2,\ldots,m\}$ is said to be a \textit{generating
partition} if the $I_{P}(x)\triangleq(P(T^{k}(x)))_{k\in\Z}$, $x\in X$, almost
surely separates points. Similarly let us call a continuous function
$f:X\rightarrow[0,1]^{d}$ a \textit{topological generator} if $I_{f}(x)\triangleq(f(T^{k}(x)))_{k\in\Z}$,
$x\in X$, separates points. Note that $d=\edim(X,T)$ is the minimal
$d$ such there exists a topological generator $f:X\rightarrow[0,1]^{d}$.
By Krieger's generator theorem (\cite{K70}) if $(X,\mathcal{B},\mu,T)$
is ergodic and $h_{\mu}(T)<\log(d)$ then there exists a generating
partition with $d$ elements.

\subsection{Dimension}

Let $\mathcal{C}$ denote the collection of open (finite) covers of
$X$. For $\alpha\in\mathcal{C}$ define its \textit{order} by $\ord(\alpha)=\max_{x\in X}\sum_{U\in\alpha}1_{U}(x)-1$.
Let $D(\alpha)=\min_{\beta\succ\alpha}ord(\beta)$ (where $\beta$
\emph{refines} $\alpha$, $\beta\succ\alpha$, if for every $V\in\beta$,
there is $U\in\alpha$ so that $V\subset U$). The Lebesgue covering
dimension is defined by $\dim(X)=\sup_{\alpha\in\mathcal{C}}D(\alpha)$.
According to the Menger-N\"{o}beling theorem (\cite{HW41}, Theorem V.2),
a compact metric space of dimension $d$ embeds in $[0,1]^{2d+1}$.

\subsection{Periodic Dimension}\label{subsec:perdim}

Let $P_{m}$ denote the set of points of period $\leq m$. Introduce the infinite vector $\overrightarrow{\perdim}(X,T)=\big(\frac{\dim(P_{m})}{m}\big)_{m\in \mathbb{N}}$. This vector is clearly a topological dynamical invariant. Let $d>0$. Denote by $\vec{d}=(d,d,\ldots)$ the infinite constant vector. We write $\overrightarrow{\perdim}(X,T)<\vec{d}$ (respectively $\overrightarrow{\perdim}(X,T)\leq\vec{d}$), or more simply $\perdim(X,T)<d$ (respectively $\perdim(X,T)\leq d$), if for every $m\in \mathbb{N}$, $\overrightarrow{\perdim}(X,T)|_m<d$ (respectively  $\overrightarrow{\perdim}(X,T)|_m\leq d$). If $X$ is finite dimensional, one can simplify matters by redefining $\perdim(X,T)=\sup_{m\in\mathbb{N}}\frac{\dim(P_{m})}{m}$, the point being that $\perdim(X,T)<d$ and $\perdim(X,T)\leq d$ have exactly the same meaning as perviously.
 It is easy to see that $\perdim(X,T)\leq \edim(X,T)$ for any t.d.s $(X,T)$.

\subsection{Mean Dimension}

Define:

\[
\mdim(X,T)=\sup_{\alpha\in\mathcal{C}}\lim_{n\rightarrow\infty}\frac{D(\alpha^{n})}{n}
\]
where $\alpha^{n}=\bigvee_{i=0}^{n-1}T^{-i}\alpha.$ Mean dimension
was introduced by Gromov \cite{G} and systematically investigated by Lindenstrauss and Weiss
in \cite{LW}. It is easy to see $\edim(X,T)\geq \mdim(X,T)$ for any
t.d.s $(X,T)$.
\begin{subrem}
\label{Rem: mdim with linear shift}For any $k\in\mathbb{N}$, $\mdim(X,T)=\sup_{\alpha\in\mathcal{C}}\lim_{n\rightarrow\infty}\frac{D(\alpha^{n+k})}{n}$.
\end{subrem}

\subsection{Lindenstrauss-Tsukamoto Embedding Conjecture}

According to Conjecture 1.2 of \cite{LT12} it should hold for all
t.d.s $(X,T)$ that if  $\mdim(X,T)<\frac{d}{2}$ and $\perdim(X,T)<\frac{d}{2}$, then $\edim(X,T)\leq d$.
This can be compared to Krieger's embedding theorem in symbolic dynamics
(\cite{K82}): If $(X,T)$ and $(Y,S)$ are irreducible shifts of
finite type so that there exists an embedding of the set of periodic points of $X$ into the set periodic points of $Y$, $(P(X),T)\hookrightarrow(P(Y),S)$
and $h_{top}(X,T)<h_{top}(Y,S)$, then there exists an embedding $(X,T)\hookrightarrow(Y,S)$.
By Flores' construction (\cite{F35}), as explained in the Introduction,
for every $n\in\mathbb{N}$, there are finite dimensional t.d.s $(X_{n},T)$ for which
$2n+1=\edim(X_{n},T)>2\perdim(X_{n},T)=2n$. By subsection 1.8 of \cite{G11}, for every $n\in\mathbb{N}$, there are t.d.s $(Y_{n},T)$ such that $2n+1=\edim(Y_{n},T)>2\mdim(Y_{n},T)=2n$. It is easy to see that for these examples it also holds $\perdim(Y_{n},T)=n$. Remarkably
by Theorem 1.3 of \cite{LT12}, for every $n\in\mathbb{N}$, there
are \textit{minimal} t.d.s $(Z_{n},T)$ for which $\edim(Z_{n},T)>2\mdim(Y_{n},T)=n$. However, unlike the condition on periodic points in Krieger's embedding theorem which is necessary for embedding, the conditions in the Lindenstrauss-Tsukamoto embedding conjecture are not necessary. That is, for some systems it holds that $\frac{d}{2}\leq\mdim(X,T)\leq d$ and $\frac{d}{2}\leq\perdim(X,T)\leq d$  and $(X,T)$  still embeds in $(([0,1]^{d})^{\mathbb{Z}},\shift)$, e.g $(X,T)=(([0,1]^{d})^{\mathbb{Z}},\shift)$.

\subsection{Jaworski's Theorems\label{sub:Jaworski's-Theorems}}

Jaworski's PhD Thesis (\cite{J74}) is the first attempt to bound
$\edim(X,T)$. Jaworski showed that if $\dim(X)=n$ and $P_{(4n+3)^{2}}=\emptyset$,
then $\edim(X,T)=1$. Coornaert (\cite{Coo05}, Chapter 8, Exercise
10) noted that it is enough to assume $P_{6n}=\emptyset$. In addition
Jaworski proved that if $X$ is a differentiable manifold of
dimension $n$ and $T:X\rightarrow X$ is a $C^{r}$-diffeomorphism ($r\geq 1$),
then every equivariant $C^{r}$-embedding of $(P_{(2n+1)^{2}},T)$
in $(([0,1])^{\mathbb{Z}},\shift)$ extends to an equivariant $C^{r}$-embedding
of $(X,T)$ in $(([0,1])^{\mathbb{Z}},\shift)$. In the literature
the name Jaworski's theorem has been attached to Theorem 9 in Chapter
13 of \cite{A}: An aperiodic finite dimensional t.d.s $(X,T)$ has
$\edim(X,T)=1$.

\subsection{Some Known Embedding Theorems}

in \cite{L99} Lindenstrauss proved that $\edim(X,T)\leq\lfloor36\mdim(X,T)\rfloor+1$
for $(X,T)$ an extension of an aperiodic minimal t.d.s. In \cite{Gut12b}
it is shown one obtains the same bound if $(X,T)$ an extension of either
an aperiodic t.d.s with a countable number of minimal subsystems or
an aperiodic finite-dimensional t.d.s. In \cite{GutTsu12} the Lindenstrauss-Tsukamoto conjecture
is verified for extensions of aperiodic subshifts (subsystems
of $\left(\{1,2,\dots,l\}^{\mathbb{Z}},\mathrm{shift}\right)$ for
some $l\in\mathbb{N}$). In \cite{G11} it was shown that
for $\mathbb{Z}^{k}$-actions ($k\in\mathbb{N})$ there exist constants
$C(k)>0$ so that if $(\mathbb{Z}^{k},X)$ is an extension a $\mathbb{Z}^{k}$ zero
dimensional t.d.s then $\edim((\mathbb{Z}^{k},X)\leq\lfloor C(k)\mdim(X,T)\rfloor+1$.

\subsection{The Baire Category Theorem Framework\label{sub:The-Baire-Category}}

All results mentioned in the previous subsection were proven through
the Baire category theorem. The same is true for the results in this
article so we now introduce the basic concepts of this technique.
A continuous mapping $f:X\rightarrow[0,1]^{d}$ induces a continuous
$\mathbb{Z}$-equivariant mapping $I_{f}:(X,T)\rightarrow(([0,1]^{d})^{\mathbb{Z}},\shift)$
given by $x\mapsto(f(T^{k}x))_{k\in\mathbb{Z}}$, also known as the
\textit{orbit-map}. Conversely, any $\mathbb{Z}$-equivariant continuous
factor map $\pi:(X,T)\rightarrow(([0,1]^{d})^{\mathbb{Z}},\shift)$
is induced in this way by $\pi_{0}:X\rightarrow[0,1]^{d}$, the projection
on the zeroth coordinate. It is therefore useful to study the space
of continuous functions $C(X,[0,1]^{d})$. Under suitable assumptions,
instead of explicitly constructing a $f\in C(X,[0,1]^{d})$ so that
$I_{f}:(X,T)\hookrightarrow(([0,1]^{d})^{\mathbb{Z}},\shift)$ is an
embedding, one shows that the property of being an embedding $I_{f}:(X,T)\hookrightarrow(([0,1]^{d})^{\mathbb{Z}},\shift)$
is \textit{generic} in $C(X,[0,1]^{d})$ (but without exhibiting an
explicit embedding). To make this precise introduce the following
definition:
\begin{subdefn}
Let $K\subset (X\times X)\setminus\triangle$ ($\triangle=\{(x,x)|\, x\in X\}$
denotes the \textit{diagonal} of $X\times X$) and $f\in C(X,[0,1]^{d})$. We say that
$I_{f}$ is \textbf{$K$-compatible} if for every $(x,y)\in K$, $I_{f}(x)\neq I_{f}(y)$,
or equivalently if for every $(x,y)\in K$, there exists $n\in\mathbb{Z}$
so that $f(T^{n}x)\neq f(T^{n}y)$. Define:
\[
D_{K}=\{f\in C(X,[0,1]^{d})|\, I_{f}\,\,\mathrm{is\,\, K}-\mathrm{compatible}\}
\]
\end{subdefn}
Under suitable assumptions one shows that for every $(x,y)\in(X\times X)\setminus\triangle$
there is an open $U\subset\overline{U}\subset(X\times X)\setminus\triangle$
with $(x,y)\in U$ so that $D_{\overline{U}}$ is open and dense in
$(C(X,[0,1]^{d}),||\cdot||_{\infty})$ (|$|\cdot||_{\infty}$ denotes
the the supremum metric, $||f-g||_{\infty}\triangleq sup_{x\in X}||f(x)-g(x)||_{\infty}$).
As $X\times X$ is second-countable, every subspace is Lindelöf, i.e.
every open cover has a countable subcover. One can therefore cover
$(X\times X)\setminus\triangle$ with a countable closed cover $\overline{U}_{1},\overline{U}_{2},\ldots$
so that $D_{\overline{U}_{m}}$ is open and dense in $(C(X,[0,1]^{d}),||\cdot||_{\infty})$
for all $m.$ By the Baire category theorem $(C(X,[0,1]^{d}),||\cdot||_{\infty})$,
is a \textit{Baire space}, i.e., a topological space where the intersection
of countably many dense open sets is dense. This implies $\bigcap_{m=1}^{\infty}D_{\overline{U}_{m}}$
is dense in $(C(X,[0,1]^{d}),||\cdot||_{\infty})$. Any $f\in\bigcap_{m=1}^{\infty}D_{\overline{U}_{m}}$
is $\overline{U}_{m}$-compatible for all $m$ simultaneously and therefore
realizes an embedding $I_{f}:(X,T)\hookrightarrow(([0,1]^{d})^{\mathbb{Z}},\shift)$.
A set in a topological space is said to be \textit{comeagre} or \textit{generic} if it
is the complement of a countable union of nowhere dense sets. A set
is said to be $G_{\delta}$ if it is the countable intersection of
open sets. As a dense $G_{\delta}$ set is comeagre, the above argument shows
that the set $\mathcal{A}\subset C(X,[0,1]^{d})$ for which $I_{f}:(X,T)\hookrightarrow(([0,1]^{d})^{\mathbb{Z}},\shift)$
is an embedding is comeagre, or equivalently, that the fact of $I_{f}$
being an embedding is generic in $(C(X,[0,1]^{d}),||\cdot||_{\infty})$.

\subsection{Overview of the article}
Section \ref{sec:Notation} contains notation. In Section \ref{sec:Periodic-Embedding} it is shown that the set of periodic points can be immersed in a suitable cubical shift whose dimension is controlled by the periodic dimension of the system. In Section \ref{sec:The-Marker-Property} the marker property is introduced. In Section \ref{sec:Bonatti-Crovisier} it is shown that aperiodic finite dimensional systems have the marker property by generalizing the Bonatti-Crovisier tower theorem. In Section \ref{sec:The Strong Topological Rokhlin Property} the strong topological Rokhlin property is introduced. In Section \ref{sec:Finite-dimensional-Jaworski-type} the Lindenstrauss-Tsukamoto conjecture is proven for finite dimensional systems. In Section \ref{sec:Infinite-dimensional-Jaworski-type} a theorem bounding the embedding dimension of an extension of an aperiodic finite dimensional system is established. The Appendix contains auxiliary lemmas.

\section{Notation\label{sec:Notation}}

In this section we gather most of the notation of the paper for the convenience of the reader.

\subsection{The system.}

Throughout the article $(X,T)$ denotes an invertible topological
dynamical system (t.d.s). $X$ is assumed to be compact and metric
with metric $d$. Denote the \emph{diagonal} of $X\times X$ by $\triangle\triangleq{\{(x,x)|\, x\in X\}}$. Let $P$ denote the set of periodic points of $(X,T)$. Let $P_{n}=\{x\in X|\,\,\exists1\leq m\leq n\,\, T^{m}x=x\}\subset P$,
be the set of periodic points of period\textbf{ $\leq n$} of $X$
and $H_{n}=P_{n}\setminus P_{n-1}$. $\dim(X)$ denotes the (Lebesgue
covering) dimension of $X$. Let $\orb(x)\triangleq\{T^{k}x\}{}_{k=-\infty}^{\infty}$. For $\alpha,\beta$ (finite) open covers of $X$, the \emph{join} of $\alpha$ and
$\beta$ is the open cover by all sets of the form $A\cap B$ where
$A\in\alpha$, $B\in\beta$. Denote $\alpha^{N}\triangleq\bigvee_{k=0}^{N-1}T^{-k}\alpha$.
For $f:X\rightarrow[0,1]^{d}$ denote by $I_{f}:(X,T)\rightarrow(([0,1]^{d})^{\mathbb{Z}},\shift)$
the $\mathbb{Z}$-equivariant induced map $x\mapsto(f(T^{k}x))_{k\in\mathbb{Z}}$.

\subsection{Distances.}

For $A,B\subset X$ closed sets denote $dist(A,B)\triangleq\min_{x\in A,y\in B}d(x,y)$.
For $A\subset X$, let $\diam(A)\triangleq\sup_{x,y\in A}d(x,y)$.
If $x\in X$ and $\epsilon>0$, let $B_{\epsilon}(x)=\{y\in X|\, d(y,x)<\epsilon\}$
denote the open ball around $x$. We denote $\overline{B}_{\epsilon}(x)=\overline{B_{\epsilon}(x)}$.
Note $\overline{B}_{\epsilon}(x)\subseteq\{y\in X|\, d(y,x)\leq\epsilon\}$
but equality does not necessary hold. Similarly let $B_{\epsilon}(F)=\{y\in X|\,\exists x\in F,\, d(y,x)<r\}$
be the ``$\epsilon$-tube'' around a set $F\subset X$. We also
define ``inward $\epsilon$-tubes'' $B_{-\epsilon}(F)=\{y\in F|\, d(y,F^{c})>\epsilon\}$.
The supremum norm in $\mathbb{R}^{l}$, $l\geq2$ is denoted by $||\cdot||_{\infty}$. $C(X,[0,1]^{d})$ stands for the collection
of continuous functions $f:X\rightarrow[0,1]^{d}$, equipped
with the supremum metric $||f-g||_{\infty}\triangleq sup_{x\in X}||f(x)-g(x)||{}_{\infty}$.
If $A\subset X$ is a closed set we denote $||f-g||_{A,\infty}\triangleq sup_{x\in A}||f(x)-g(x)||{}_{\infty}$. If $A=X$ one simply writes $||f-g||_{\infty}$.

\subsection{Tuples.}

For $V=\{v_{1},\ldots,v_{m}\}\subset\mathbb{R}^{n}$, denote by
$\co{(V)}\triangleq{\{\sum_{i=1}^{m}\lambda_{i}v_{i}|\ \sum_{i=1}^{m}\lambda_{i}=1,\lambda_{i}\geq0\}}$ the \textit{convex hull} of $V$.
Let $K$ be a topological space. For a tuple $v=(v_{i})_{i=0}^{n-1}\in K^{n}$
and $0\leq a\leq b\leq n-1$, we let $v|_{a}^{b}\triangleq(v_{i})_{i=a}^{b}$
and $v|_{a}\triangleq v_{a}$. If $I=\{i_{1},\ldots,i_{l}\}\subset{\{0,\ldots,n-1\}}$
with $i_{1}<\cdots<i_{l}$, we let $v|_{I}\triangleq(v_{i_{k}})_{k=1}^{l}$.
If $n\geq m$ and $v=(v_{i})_{i=0}^{m-1}\in K^{m}$, then $v^{\oplus n}\in K^{n}$
is defined by $v^{\oplus n}|_{k}\triangleq v|_{(k\mod m)}$ and for
$1\leq l\leq m-1$, $v^{\bullet l}\in K^{m}$ is defined by $v^{\bullet l}|_{k}\triangleq v_{(k+l\mod m)}$.
A statement of the form ``Property \textbf{P} holds for $(v_{1},v_{2},\ldots,v_{m})\in([0,1]{}^{n})^{m}$
almost surely (w.r.t Lebesgue measure)'' is to be understood in the
following way: $\mu(\{(v_{ij})\in[0,1]^{n}{}^{m}|\,\mathbf{P\textrm{ holds for }}(v_{ij})\}=1$
where $\mu$ is the Lebesgue measure of $[0,1]^{n}{}^{m}$.

\subsection{Miscellaneous }

If $\mathcal{U}={\{C_{1},C_{2},\ldots,C_{n}\}}$, then $\bigcup\mathcal{U}\triangleq C_{1}\cup C_{2}\cup\cdots C_{n}$.
Given an open cover $\alpha$ of $X$ and $\{\psi_{U}\}_{U\in\alpha}$ a
partition of unity subordinate to $\alpha$, define for $x\in X$,
$\alpha_{x}=\{U\in\alpha|\,\psi_{U}(x)>0\}$. Let $\mathbb{R}_{+}={\{x\in\mathbb{R}|\, x\geq0\}}$
and $\mathbb{Z}_{+}=\mathbb{N}\cup\{0\}$. For $r\in\R$, $[r]=\lfloor r\rfloor=\max\{n\in\Z\,|\, n\leq r\}$
denotes the floor function. $\lceil r\rceil=\min\{n\in\Z\,|\, n\geq r\}$
denotes the ceiling function. ${\{r\}}=r-\lfloor r\rfloor$ denotes
the fractional part function. Let $N\in\mathbb{N}$. Denote $\mathbb{Z}_{N}=\mathbb{Z}/N\mathbb{Z}$.
Addition in $\mathbb{Z}_{N}$ is denoted by $a\oplus b$.

\section{Immersion of the set of  periodic points}\label{sec:Periodic-Embedding}
\begin{thm}
\label{thm:periodic case of embedding conjecture}Assume $\perdim(X,T)<\frac{d}{2}$
for some $d\in\mathbb{N}$. Then the set of continuous mappings $f:X\rightarrow[0,1]^{d}$
so that $I_{f}:(P,T)\hookrightarrow(([0,1]^{d})^{\mathbb{Z}},\shift)$
is an immersion is comeagre in $C(X,[0,1]^{d})$.\end{thm}
\begin{proof}
Note $(P\times P)\setminus\triangle=\big((\stackrel{\circ}{\bigcup}_{n\in\mathbb{N}}H_{n})\times(\stackrel{\circ}{\bigcup}_{n\in\mathbb{N}}H_{n})\big)\setminus\triangle=\big(\stackrel{\circ}{\bigcup}_{n\in\mathbb{N}}H_{n}\times H_{n})\setminus\triangle\big)\cup\stackrel{\circ}{\bigcup}_{n\neq m}H_{n}\times H_{m}$, where $H_{n}$ is the set of $x\in X$, whose minimal period is $n$ ($H_{n}=P_{n}\setminus P_{n-1}$). Let $n\in\mathbb{N}$ and fix distinct $x,y\in H_{n}$. Either  $x,Tx,\ldots,T^{n-1}x,y,Ty,\ldots,T^{n-1}y$ are pairwise distinct or $y=T^{k}x$ for some $1\leq k\leq n-1$. Thus by Propositions \ref{Prop: U_x times U_y is dense} and \ref{Prop:periodic} below one may find open neighborhoods in $H_{n}$ (but
not necessarily open in $X$) $R_x$ and $R_y$ of $x$ and $y$  respectively so that for $K_{(x,y)}=\overline{R}_{x}\times\overline{R}_{y}$,
$D_{K_{(x,y)}}$ is dense in $C(X,[0,1]^{d})$ and $K_{(x,y)}\cap\triangle=\emptyset$.
As $X$ is second-countable, every subspace is a Lindelöf space, i.e.
every open cover has a countable subcover. We can therefore choose an open
countable cover of $\stackrel{\circ}{\bigcup}_{n\in\mathbb{N}}(H_{n}\times H_{n})\setminus\triangle$
by $\mathcal{U}=\{R_{x_{1}}\times R_{y_{1}},R_{x_{2}}\times R_{y_{2}},\ldots\}$. Thus there is a closed countable cover of $\stackrel{\circ}{\bigcup}_{n\in\mathbb{N}}(H_{n}\times H_{n})\setminus\triangle$
by $\overline{\mathcal{U}}\triangleq\{K_{(x_{1},y_{1})},K_{(x_{2},y_{2})},\ldots\}$. Now let $n,m\in\mathbb{N}$ with $n\neq m$ and fix $x\in H_{n}$ and $y\in H_{m}$. Note $x,Tx,\ldots,T^{n-1}x,y,Ty,\ldots,T^{m-1}y$ are pairwise distinct and thus yet again by Proposition \ref{Prop: U_x times U_y is dense}  below one may find open neighborhoods in $H_{n}$ (but
not necessarily open in $X$) $R_x$ and $R_y$ of $x$ and $y$  respectively so that for $K_{(x,y)}=\overline{R}_{x}\times\overline{R}_{y}$,
$D_{K_{(x,y)}}$ is dense in $C(X,[0,1]^{d})$ and $K_{(x,y)}\cap\triangle=\emptyset$. Similarly to above choose a closed
countable cover $\overline{\mathcal{V}}\triangleq\{K_{(x'_{1},y'_{1})},K_{(x'_{2},y'_{2})},\ldots\}$ of $\stackrel{\circ}{\bigcup}_{n\neq m}H_{n}\times H_{m}$. By Lemma \ref{lem:D_K is open} for
any $(x,y)\in (P\times P)\setminus\triangle$, $D_{K_{(x,y)}}$ is open in $C(X,[0,1]^{d})$. Thus there is a closed countable cover of $(P\times P)\setminus\triangle$
by closed sets $\overline{\mathcal{U}}\cup \overline{\mathcal{V}}=\{K_1,K_2,\ldots\}$
so that $D_{K_i}$ is open and dense in $C(X,[0,1]^{n})$
for $i\in\mathbb{N}$.
Define the comeagre set $\mathcal{A}=\bigcap_{K\in
\overline{\mathcal{U}}\cup \overline{\mathcal{V}}}D_{K}$.
Clearly for any $f\in\mathcal{A}$, $I_{f}:(P,T)\hookrightarrow(([0,1]^{d})^{\mathbb{Z}},\shift)$ is injective
(see Subsection \ref{sub:The-Baire-Category} for more explanations).
\end{proof}

\begin{prop}
\label{Prop: U_x times U_y is dense} Let $n_1,n_2\in \mathbb{N}$.
Given a periodic points $x\in X$ whose minimal period is $n_1$ and $y\in X$ whose minimal period is $n_2$ such that $x,Tx,\ldots,T^{n_1-1}x,y,Ty,\ldots,T^{n_2-1}y$ are pairwise distinct,
there exist open neighborhoods $R_1$ and $R_2$ of $x$ and $y$ respectively such that
for
$K_{(x,y)}=\overline{R}_{1}\times\overline{R}_{2}$, the set of $K_{(x,y)}$-compatible functions
is dense in $C(X,[0,1]^{d})$ and $K_{(x,y)}\cap\triangle=\emptyset$.
\end{prop}
\begin{proof}
As $x,Tx,\ldots,T^{n_1-1}x,y,Ty,\ldots,T^{n_2-1}y$ are pairwise distinct, one may select $\epsilon>0$ such that $\overline{B}_{\epsilon}(x),T\overline{B}_{\epsilon}(x),\ldots,T^{n_1-1}\overline{B}_{\epsilon}(x),\overline{B}_{\epsilon}(y),T\overline{B}_{\epsilon}(y),\ldots,T^{n_2-1}\overline{B}_{\epsilon}(y)$ are pairwise disjoint. Define $R_{1}=B_{\epsilon}(x)\cap H_{n_1}$ and $R_{2}=B_{\epsilon}(y)\cap H_{n_2}$. Notice that $R_{1}$ and $R_{2}$  are open in $H_{n_1}$ and $H_{n_2}$ respectively (but
not necessarily open in $X$). We assume $\epsilon$ to be small enough such that
$\overline{R}_{1}$ and $\overline{R}_{2}$ (closure in $X$) are contained in $H_{n_1}$ and $H_{n_2}$ respectively. E.g.
if $d(x,P_{n-1})=r>0$, and $0<\epsilon<r$ then
$\overline{R}_{1}=\overline{B}_{\epsilon}(x)\cap P_{n_1}$ obeys $\overline{R}_{1}=\overline{B}_{\epsilon}(x)\cap H_{n_1}$.

Assume w.l.o.g $n_{1}\geq n_{2}$.
 Fix $\epsilon>0$. Let $\tilde{f}:X\rightarrow[0,1]^{d}$
be a continuous function. We will show that there exists a continuous
function $f:X\rightarrow[0,1]^{d}$ so that $\|f-\tilde{f}\|_{\infty}<\epsilon$
and $I_{f}$ is $(\overline{R}_{1}\times\overline{R}_{2})$-compatible.
As $\frac{\ensuremath{\dim(\overline{R}_{i})}}{n_{i}}<\frac{d}{2}$,
one can choose open covers $\alpha_{i}$ of $\overline{R}_{i}$ with
$ord(\alpha_{i})<\frac{n_{i}d}{2}$ so that $\max_{U\in\alpha_{i},0\leq k\leq n_{i}-1}\diam(\tilde{f}(T^{k}U))<\frac{\epsilon}{2}$. For each $U\in\alpha_{i}$ choose $q_{U}\in U$, so that $\{q_{U}\}_{U\in \alpha_i}$  is a collection of distinct points in $\overline{R}_i$ and define $\tilde{v}_{U}=(\tilde{f}(T^{j}q_{U}))_{j=0}^{n_{i}-1}$.
As $\overline{R}_{1},\overline{R}_{2}$
are disjoint, by Lemma \ref{lem:F periodic disjoint case}
one can find continuous functions $F_{i}:\overline{R}_{i}\rightarrow([0,1]^{d})^{n_{i}}$,
with the following properties ($i=1,2$):
\begin{enumerate}
\item \label{enu:Very Close-1}$\forall U\in\alpha_{i}$, $||F_{i}(q_{U})-\tilde{v}_{U}||_{\infty}<\frac{\epsilon}{2}$,
\item \label{enu:convex combination-1}$\forall x\in X$, $F_{i}(x)\in \co\{F(q_{U})|\, x\in U\in\alpha_{i}\}$,
\item \label{enu:F(x)neqF oplus n (y)-1}If $x'\in\overline{R}_{1}$ and
$y'\in\overline{R}_{2}$ then $F_{1}(x')\neq(F_{2}(y'))^{\oplus n_{1}}$
(recall that for $v=(v_{i})_{i=0}^{n_{2}-1}\in([0,1]^{d})^{n_{2}}$,
$v^{\oplus n_{1}}\in([0,1]^{d})^{n_{1}}$ is defined by $v^{\oplus n_{1}}|_{k}=v|_{(k\mod n_{2})}$, $0\leq k\leq n_{1}-1$).
\end{enumerate}
\noindent
Using that $\overline{R}_{1},T\overline{R}_{1},\ldots,T^{n_{1}-1}\overline{R}_{1},\overline{R}_{2},T\overline{R}_{2},\ldots,T^{n_{2}-1}\overline{R}_{2}$
are pairwise disjoint, define for $z\in\overline{R}_{i}$ and $0\leq k\leq n_{i}-1$:
\[
f_{|T^{k}\overline{R}_{i}}(T^{k}z)\triangleq F_{i}(z)|_{k}
\]
Notice that for $z\in\overline{R}_{2}$ and $0\leq k\leq n_{1}-1$
it holds:

\[
f_{|T^{_{k}}\overline{R}_{2}}(T^{k}z)=(F_{2}(z))^{\oplus n_{1}}|_{k}
\]
Indeed if $k_{1}=k_{2}\mod n_{2}$, $0\leq k_{1}\leq n_{1}-1$ and
$0\leq k_{2}\leq n_{2}-1$, then $T^{k_{1}}z=T^{k_{2}}z$
and $f_{|T^{_{k_{1}}}\overline{R}_{2}}(T^{k_{1}}z)=f_{|T^{_{k_{2}}}\overline{R}_{2}}(T^{k_{2}}z)=(F_{2}(z))^{\oplus n_{1}}|_{k_{1}}$
by definition.

Define $B=\bigcup_{k=0}^{n_{1}-1}T^{k}\overline{R}_{1}\cup\bigcup_{k=0}^{n_{2}-1}\overline{R}_{2}$.
As by property (\ref{enu:convex combination-1}) above, for
$z\in\overline{R}_{i}$ and $0\leq k\leq n_{i}-1$, $f(T^{k}z)=F_i(z)_{|k}\in \co\{F_i(q_{U})_{|k}|\, z\in U\in\alpha_{i}\}$, one has $||f(T^{k}z)-\tilde{f}(T^{k}z)||_{\infty}\leq\max_{z\in U\in\alpha_{i}}||F_{i}(q_{U})|_{k}-\tilde{f}(T^{k}z)||_{\infty}$.
Fix $U\in\alpha_{i}$ and $z\in U$. Note $||F_{i}(q_{U})|_{k}-\tilde{f}(T^{k}z)||_{\infty}\leq||F_{i}(q_{U})|_{k}-\tilde{v}_{U}|_{k}||_{\infty}+||\tilde{v}_{U}|_{k}-\tilde{f}(T^{k}z)||_{\infty}$.
The first term on the right-hand side is bounded by $\frac{\epsilon}{2}$
by property (\ref{enu:Very Close-1}). As $\diam(\tilde{f}(T^{k}U))<\frac{\epsilon}{2}$
and $\tilde{v}_{U}|_{k}=\tilde{f}(T^{k}q_{U})$ we have $||\tilde{f}(T^{k}q_{U})-\tilde{f}(T^{k}z)||_{\infty}<\frac{\epsilon}{2}$. We conclude $||\tilde{f}_{|B}-f||_{B,\infty}<\epsilon.$ By Lemma
\ref{lem:Tietze Extension}, we can extend $f$ continuously to $f:X\rightarrow[0,1]^{d}$
so that $||\tilde{f}-f||_{\infty}<\epsilon$.

Assume for a contradiction $I_{f}(x')=I_{f}(y')$ for some $(x',y')\in\overline{R}_{1}\times\overline{R}_{2}$.
In particular we have $F_{1}(x')=(f(x'),\ldots,f(T^{n_{1}-1}x'))=(f(y'),\ldots,f(T^{n_{1}-1}y'))=(F_{2}(y'))^{\oplus n_{2}}$
which is a contradiction to property (\ref{enu:F(x)neqF oplus n (y)-1}).
\end{proof}

\begin{prop}
\label{Prop:periodic} Let $n\in \mathbb{N}$.
Given a periodic points $x\in X$ whose minimal period is $n$ and $y\in X$  such that $y=T^{l}x$ for some $1\leq l\leq n-1$,
there exist an open neighborhood $R$ of $x$ such that
for
$K_{(x,y)}=\overline{R}\times T^l\overline{R}$, the set of $K_{(x,y)}$-compatible functions
is dense in $C(X,[0,1]^{d})$ and $K_{(x,y)}\cap\triangle=\emptyset$.
\end{prop}
\begin{proof}

As $x,Tx,\ldots,T^{n-1}x$ are pairwise distinct, one may select $\epsilon>0$ such that $\overline{B}_{\epsilon}(x),T\overline{B}_{\epsilon}(x),\ldots,T^{n-1}\overline{B}_{\epsilon}(x)$ are pairwise disjoint. Define $R=B_{\epsilon}(x)\cap H_{n}$. As in Proposition \ref{Prop: U_x times U_y is dense} we may assume $\epsilon$ is small enough such that
$\overline{R}$ (closure in $X$) is contained in $H_{n}$.
  As $\frac{\ensuremath{\dim(\overline{R})}}{n}<\frac{d}{2}$,
one can choose an open cover $\alpha$ of $\overline{R}$ with
$ord(\alpha)<\frac{nd}{2}$ so that $\max_{U\in\alpha,0\leq k\leq n-1}\diam(\tilde{f}(T^{k}U))<\frac{\epsilon}{2}$. For each $U\in\alpha$ choose $q_{U}\in U$, so that $\{q_{U}\}_{U\in \alpha}$  is a collection of distinct points in $\overline{R}$ and define $\tilde{v}_{U}=(\tilde{f}(T^{j}q_{U}))_{j=0}^{n-1}$.
By Lemma \ref{lem:F periodic translation case }
one can find a continuous function $F:\overline{R}\rightarrow([0,1]^{d})^{n}$,
with the following properties:
\begin{enumerate}
\item $\forall U\in\alpha,\ ||F(q_{U})-\tilde{v}_{U}||_{\infty}<\frac{\epsilon}{2}$,
\item \label{enu:convex combination-1-1}$\forall x\in X$, $F(x)\in \co\{F(q_{U})|\, x\in U\in\alpha\}$,
\item \label{enu:F(x)neqF(y)-1-1}If $x',y'\in\overline{R}$ then $F(x')\neq(F(y'))^{\bullet l}$
(recall that for $v=(v_{i})_{i=0}^{n-1}\in([0,1]^{d})^{n}$ and $1\leq l\leq n-1$,
$v^{\bullet l}|_{k}=v_{(k+l)\mod n}$, $0\leq k\leq n-1$).
\end{enumerate}
\noindent
Using that
$\overline{R},T\overline{R},\ldots,T^{n-1}\overline{R}$ are pairwise disjoint, define for $z\in\overline{R}$ and $0\leq k\leq n-1$:
\begin{equation}\label{eq:f is F}
f_{|T^{k}\overline{R}}(T^{k}z)\triangleq F(z)|_{k}
\end{equation}
\noindent
Denote $\overline{R}_{2}=T^{l}\overline{R}$ and notice that for $\tilde{z}\in\overline{R}_{2}$ and $0\leq k\leq n-1$
it holds:
\begin{equation}\label{eq:f is F+l}
f_{|T^{_{k}}\overline{R}_{2}}(T^{k}\tilde{z})=(F(T^{-l}\tilde{z}))^{\bullet l}|_{k}
\end{equation}
Indeed $f_{|T^{_{k}}\overline{R}_{2}}(T^{k}\tilde{z})=f_{|T^{_{k+l}}\overline{R}}(T^{k+l}(T^{-l}\tilde{z}))=(F(T^{-l}\tilde{z}))^{\bullet l}|_{k}$
where we have used the fact that for any $k\in\mathbb{N}$, as $R\subset H_{n}$,
$f_{|T^{_{k}}\overline{R}}(T^{k}z))=F(z)|_{k\mod n}$. As in Proposition \ref{Prop: U_x times U_y is dense} one can extend $f$ continuously to $f:X\rightarrow[0,1]^{d}$
with $||\tilde{f}-f||_{\infty}<\epsilon$.

Assume for a contradiction $I_{f}(x')=I_{f}(y')$ for some $(x',y')\in\overline{R}\times\overline{R}_{2}$.
In particular we have by (\ref{eq:f is F}) and (\ref{eq:f is F+l}), $F(x')=(f(x'),\ldots,f(T^{n-1}x'))=(f(y'),\ldots,f(T^{n-1}y'))=(F(T^{-l}y'))^{\bullet l}$
which is a contradiction to property (\ref{enu:F(x)neqF oplus n (y)-1}).
\end{proof}

\section{The Marker Property}\label{sec:The-Marker-Property}
\begin{defn}
A subset $F$ of a topological dynamical system $(X,T)$ is called
an \textbf{$n$-marker} ($n\in\mathbb{N}$) if:
\begin{enumerate}
\item $F\cap T^{i}(F)=\emptyset$ for $i=1,2,\ldots,n-1$.
\item The sets $\{T^{i}(F)\}_{i=1}^{m}$ cover $X$ for some $m\in\mathbb{N}$.
\end{enumerate}
The system $(X,T)$ is said to have the \textbf{marker property} if
there exist \textit{open} $n$-markers for all $n\in\mathbb{N}$.
\end{defn}
\begin{rem}
\label{Rem: marker property stable under extension}Clearly the marker
property is stable under extension, i.e. if $(X,T)$ has the marker
property and $(Y,S)\rightarrow(X,T)$ is an extension, then $(Y,S)$
has the marker property.
\end{rem}

\begin{rem}
By Lemma \ref{lem:closed marker iff open} we can (and do sometimes throughout the article) consider closed
$n$-markers instead of open $n$-markers without loss of generality.
\end{rem}

The marker property was first defined in \cite[Definition 2]{Dow06}, where one requires the $n$-markers to be clopen. In the same
article it was proven that an extension of an aperiodic
zero-dimensional (non necessarily invertible) t.d.s has the marker
property. This was essentially based on the "Krieger Marker Lemma"
(\cite[Lemma 2]{K82}). From \cite[Lemma 3.3]{L99} it follows that
an extension of an aperiodic minimal system has the marker property.
In \cite{Gut12b} the marker property and the related \textit{local}
marker property are investigated. In the same article, some classes of
t.d.s having the marker property are exhibited, for example, extensions
of aperiodic t.d.s with a countable number of minimal subsystems. We
mention an open problem:
\begin{problem}
\label{Ques: Aperiodic-->Marker?}Does every aperiodic system have the
marker property?
\end{problem}

\section{Aperiodic finite dimensional systems have the marker property\label{sec:Bonatti-Crovisier}}

 We present a partial generalization of the Bonatti-Crovisier tower
theorem (\cite[Theorem 3.1]{BC04}), proven for
$C^{1}$-diffeomorphism on manifolds,
to the setting of homeomorphisms of metric compacta:
\begin{thm}
\label{thm:finite dimensional has local marker property}Let $(X,T)$
be an aperiodic finite dimensional t.d.s, then $(X,T)$ has the marker
property.
\end{thm}
In order to prove the preceding theorem, we generalize Lemma 3.7 on p. 60 of \cite{BC04}. We do not assume $(X,T)$ is aperiodic but only that $P$ is closed. This is exploited in \cite{Gut12b}. For the purposes of Theorem \ref{thm:finite dimensional has local marker property} the reader is of course free to assume $P=\emptyset$.

\begin{lem}
\label{lem:disjontifying_two_sets}Let $X$ be a compact metric space
with $\dim(X)=d$ with a closed set of periodic points $P$. Fix $N\in\mathbb{N}$. Let $U$ and $V$ be two open sets
in $X$ with $\overline{U},\overline{V}\subset X\setminus P$. Then there exists $m=m(d,N)=(2d+2)N-1\in\mathbb{N}$ such that
if\end{lem}
\begin{enumerate}
\item $\overline{U}\cap T^{i}\overline{U}=\emptyset$ for $i=1,2,\ldots,N-1$.
\item $\overline{V}\cap T^{i}\overline{V}=\emptyset$ for $i=1,2,\ldots,m$.
\end{enumerate}
then there exists an open $W$ so that $\overline{U}\subset\overline{W}$, $\overline{W}\subset X\setminus P$,
$\overline{V}\subset\bigcup_{i=1}^{m}T^{i}(\overline{W})$ and $\overline{W}\cap T^{i}\overline{W}=\emptyset,\, i=1,2,\ldots,N-1$.
\begin{proof}[Proof of Theorem \ref{thm:finite dimensional has local marker property}
using Lemma \ref{lem:disjontifying_two_sets}] Cover $X$ by $U_{1},U_{2},\ldots,U_{s}$
so that $\overline{U}_{j}\cap T^{i}\overline{U}_{j}=\emptyset,\, i=1,2,\ldots m(d,N)$,
$j=1,2,\ldots,s$ . Apply Lemma \ref{lem:disjontifying_two_sets}, to get an open set $W_2$ so that $\overline{W}_2\cap T^{i}\overline{W}_2=\emptyset,\, i=1,2,\ldots,N-1$,
 $\overline{U}_1\subset\overline{W}_2$ and $\overline{U}_2\subset\bigcup_{i=1}^{m}T^{i}(\overline{W}_2)$. Proceed by induction. Assume there is an open set $W_k$ so that $\overline{W}_k\cap T^{i}\overline{W}_k=\emptyset,\, i=1,2,\ldots,N-1$ and $\bigcup_{i=1}^{k}\overline{U}_{i}\subset\bigcup_{i=0}^{m}T^{i}(\overline{W}_k)$. By Lemma  \ref{lem:disjontifying_two_sets}, one can find an open set $W_{k+1}$ so that $\overline{W}_{k+1}\cap T^{i}\overline{W}_{k+1}=\emptyset,\, i=1,2,\ldots,N-1$, $\overline{W}_k\subset \overline{W}_{k+1}$ and $\overline{U}_{k+1}\subset\bigcup_{i=1}^{m}T^{i}(\overline{W}_{k+1})$. The last two inclusions imply $\bigcup_{i=1}^{k+1}\overline{U}_{i}\subset\bigcup_{i=0}^{m}T^{i}(\overline{W}_{k+1})$. We thus conclude that for $W\triangleq W_s$, one has  $\overline{W}\cap T^{i}\overline{W}=\emptyset,\, i=1,2,\ldots,N-1$ and $X=\bigcup_{i=1}^{s}\overline{U}_{i}\subset\bigcup_{i=0}^{m}T^{i}(\overline{W})$.
\end{proof}
In the proof of the preceding lemma, we are going to use the following
definition from p. 944 of \cite{Ku95}:
\begin{defn}
Let $X$ be a compact space with $\dim(X)=d$ and $l\in\mathbb{N}$.
We say $S_{1},S_{2},\ldots,S_{l}\subset X$ are in \textbf{general
position} if for all $\mathcal{I}\subset{\{1,2,\ldots,l\}}$ with
$|\mathcal{I}|=m$, one has $\dim(\bigcap_{i\in\mathcal{\mathcal{I}}}S_{i})\leq\max\{-1,d-m\}$
(where $\dim(Y)=-1$ iff $Y=\emptyset$).

The following lemma is a direct consequence of Lemma 3.5 of \cite{L95}:
\begin{lem}
\label{lem:general position of translations}Let $X$ be a compact metric space
with $\dim(X)=d$ and with a closed set of periodic points $P$. Let $U$
be an open set so that $\overline{U}\subset X\setminus P$. For every $k\in\mathbb{N}$ and $V$ open set with
\textup{$\partial U\subset V$} there exists an open set $U'$ with
$U\subset U'\subset U\cup V$ and $\partial U'\subset\partial U\cup V$
such that $\{\partial U',T\partial U',\ldots T^{k-1}\partial U'\}$
are in general position.
\end{lem}
\end{defn}
\begin{rem}
In the sequel we only need the weaker property that any subcollection
of $\{\partial U',T\partial U',\ldots T^{k-1}\partial U'\}$ of cardinality
$\geq d+1$ has empty intersection. \end{rem}
\begin{proof}[Proof of Lemma \ref{lem:disjontifying_two_sets}]

The idea of the proof is the following: If $\overline{U}$ and $\overline{V}$ have both $m$ disjoint iterates, it does not necessarily hold
for their union $\overline{U}~\cup~\overline{V}$, which is the naive
candidate for $\overline{W}$. Instead of this union we decompose
$\overline{V}$ into pieces and construct a closed set $\overline{W}$
containing $\overline{U}$ and (different) iterates of these pieces.
Specifically denote $R\triangleq\overline{V}\setminus\bigcup_{i=1}^{m}T^{i}U$.
We will define a finite collection of closed sets  $\overline{\mathcal{U}}$ which
covers $R$ and has some additional properties that will be essential
for the definition of $W$. Let $\rho>0$ be such that $\overline{B}_{\rho}(R)\cap T^{i}\overline{B}_{\rho}(R)=\emptyset,\ i=1,2,\ldots,m$ and $\overline{B}_{\rho}(R)\subset X\setminus P$.
By Lemma \ref{lem:general position of translations}, by enlarging
$U$ one can assume w.l.o.g that $\mathcal{B}\triangleq\{\partial U,T^{1}\partial U,\ldots,T^{m}\partial U\}$
is in general position and still $\overline{U}\cap T^{i}\overline{U}=\emptyset$,
$i=1,2,\ldots,N-1$ and $\overline{U}\subset X\setminus P$.

We claim that there is $0<\delta<\rho$ so that for every $x\in R$,
$|\{i\in\{1,\ldots,m\}|\,\, T^{i}\overline{U}\cap\overline{B}_{\delta}(x)\neq\emptyset\}|\leq d$.
Indeed assume not. Then one can find sequences $x_{k}\in R$, $\delta_{k}>0$,
$k\in\mathbb{N}$ so that $\delta_{k}\rightarrow_{k\rightarrow \infty}0$ and $x_{k}\rightarrow_{k\rightarrow \infty} x$
for some $x\in R$ (as $R$ is closed) with $|\{i\in\{1,\ldots,m\}|\, T^{i}\overline{U}\cap\overline{B}_{\delta_{k}}(x_{k})\neq\emptyset\}|>d$.
By refining the sequence one can find distinct $i_{1},\ldots i_{d+1}\in \{1,\ldots,m\}$
and $y_{k}^{l}\in X$, $k\in I\subset\mathbb{N}$, $|I|=\infty$,
$l\in\{1,2,\ldots,d+1\}$ so that $d(y_{k}^{l},x_{k})\leq\delta_{k}$
and $y_{k}^{l}\in T^{i_{l}}\overline{U}$. As for every $l$, $y_{k}^{l}\rightarrow_{k\rightarrow \infty} x$,
we conclude $x\in R\cap\bigcap_{l=1}^{d+1}T^{i_{l}}\overline{U}\subset\bigcap_{l=1}^{d+1}T^{i_{l}}\partial\overline{U}$.
This contradicts $\mathcal{B}$ being in general position. Choose
an open finite cover of $R$ of the form $\mathcal{\mathcal{U}}=\{B_{\delta}(z_{i})\}_{i=1}^{s}$
($z_{i}\in R)$ and define $\overline{\mathcal{U}}=\{\overline{M}|\, M\in\mathcal{U}\}$.
We record the crucial property of the closed cover $\overline{\mathcal{U}}$:

\begin{equation}
\forall E\in\overline{\mathcal{U}}\:|\{i\in\{1,\ldots,m\}|\, T^{i}\overline{U}\cap E\neq\emptyset\}|\leq d\label{eq:d_intersections_at_most}
\end{equation}
Let $O=\bigcup\overline{\mathcal{U}}$. Clearly $R\subset O$. Note
$O\subset\overline{B}_{\rho}(R)$, which implies:

\begin{equation}
O\cap T^{i}O=\emptyset,\ i=1,2,\ldots,m\label{eq:O_Separation-1}
\end{equation}
We now associate to each element in $\overline{\mathcal{U}}$ an integer
in ${\{1,2,\dots,2d+1\}}$, i.e. we define a mapping $c:\overline{\mathcal{U}}\rightarrow{\{1,2,\dots,2d+1\}}$.
Fix $E\in\overline{\mathcal{U}}$. We claim there is at least one
element $c(E)$ in ${\{1,2,\dots,2d+1\}}$, so that for $t\in{\{-(N-1),\ldots,(N-1)\}}$,
$T^{t}E\cap T^{c(E)N}\overline{U}=\emptyset$. Indeed start out by
marking the indices $1\leq i_{1}<i_{2}<\cdots<i_{s}\leq m$, so that
$E\cap T^{i_{k}}\overline{U}\neq\emptyset$. By Equation (\ref{eq:d_intersections_at_most}),
$s\leq d$. Consider the index intervals $I_{k}=\{-N+1+kN,\ldots,kN+N-1\}$,
$k=1,2,\ldots,2d+1$. As $m=(2d+1)N+N-1$, $I_{k}\subset\{1,\ldots,m\}$.
Clearly each $i_{l}$ can be contained in at most two index intervals.
By a simple pigeonhole argument there must be an interval $I_{c(E)}$
which does not contain any of the $i_{l}$, $l=1,\ldots,s$. This is
equivalent to the desired statement.

Define the open set $W=U\cup\bigcup_{H\in\mathcal{\mathcal{U}}}T^{-c(\overline{H})N}H$
and note:

\[
\overline{W}=\overline{U}\cup\bigcup_{E\in\overline{\mathcal{U}}}T^{-c(E)N}E
\]
Clearly $\overline{U}\subset\overline{W}$, $\overline{V}\subset\bigcup_{i=1}^{m}T^{i}(\overline{W})$ and $\overline{W}\subset X\setminus P$. Indeed let $x\in \overline{V}$. If $x\in \bigcup_{i=1}^{m}T^{i}U$, then certainly $x\in  \bigcup_{i=1}^{m}T^{i}(\overline{W})$. Otherwise $x\in R=\overline{V}\setminus \bigcup_{i=1}^{m}T^{i}U$. Thus there exist $E\in \overline{\mathcal{U}}$, such that $x\in E\subset T^{c(E)N}\overline{W}$. Clearly $1\leq c(E)N< m$
We now claim $\overline{W}\cap T^{i}\overline{W}=\emptyset,\, i=1,2,\ldots,N-1$.
Assume not, then there exists $x,y\in W$ and $t\in\{1,2,\ldots,N-1\}$
so that $x=T^{t}y$. We consider several cases:
\begin{enumerate}
\item $x,y\in\overline{U}$. Contradicts the fact $\overline{U}\cap T^{i}\overline{U}=\emptyset$,
$i=1,2,\ldots,N-1$.
\item $x,y\in T^{-c(E)N}E\subset T^{-c(E)N}O$ for some $E\in\overline{\mathcal{U}}$. Contradicts Equation (\ref{eq:O_Separation-1}).
\item $x\in T^{-c(E_{1})N}E_{1}$ and $y\in T^{-c(E_{2})N}E_{2}$ for $E_{1}\neq E_{2}$.
Conclude $T^{t+c(E_{1})N}y\in E_{1}\subset O$ and $T^{c(E_{2})N}y\in E_{2}\subset O$.
However as $0<|t+c(E_{1})N~-~c(E_{2})N|\leq(2d+1)N<m$ this contradicts Equation (\ref{eq:O_Separation-1}).
\item $x\in T^{-c(E)N}E$ and $y\in\overline{U}$. Conclude $E\cap T^{t+c(E)N}\overline{U}\neq\emptyset$.
This contradicts the definition of $c(E)$.
\item $x\in\overline{U}$ and $y\in T^{-c(E)N}E$. Conclude $E\cap T^{-t+c(E)N}\overline{U}\neq\emptyset$.
This contradicts the definition of $c(E)$.
\end{enumerate}
\end{proof}

\section{The Strong Topological Rokhlin Property}\label{sec:The Strong Topological Rokhlin Property}

In \cite[Subsection 1.9]{G11}, the topological Rokhlin property was
introduced. Here is a stronger variant, originating in \cite{L99}, which is further discussed
in \cite{Gut12b}:
\begin{defn}
We say that $(X,T)$ has the \textbf{(global) strong topological Rokhlin
property} if for every $n\in\mathbb{N}$ there exists a continuous
function $f:X\rightarrow\mathbb{R}$ so that if we define the \textit{exceptional
set} $E_{f}=\{x\in X\,|\, f(Tx)\neq f(x)+1\}$, then $E_{f}\cap T^{i}(E_{f})=\emptyset$
for $i=1,\ldots n-1$.\end{defn}
\begin{rem}
\label{rem:at most one bad index (global case)}Let $a,b\in\mathbb{Z}$
with $b-a\leq n-1$. Under the above assumptions consider $x\in X$.
Then there exists at most one index $a\leq l\leq b$ so that $f(T^{l+1}x)\neq f(T^{l}x)+1$.
Indeed if $T^{l}x,T^{l'}x\in E_{f}$ for $a\leq l<l'\leq b$, then
$E_{f}\cap T^{l'-l}E_{f}\neq\emptyset$ contradicting the definition.
\end{rem}
The following theorem is proven in \cite{Gut12b} but we repeat it for the convenience of the reader. The proof of the direction $\Leftarrow$ is based on \cite[Lemma 3.3]{L99}:

\begin{thm}\label{thm:-Marker Property implies Strong Rokhlin}$(X,T)$
has the strong topological Rokhlin property iff $(X,T)$ has the marker
property.
\end{thm}

\begin{proof}
Assume $(X,T)$ has the marker
property and fix $m\in \mathbb{N}$. One may choose a closed set $R$ and an open set $F$ with $R\subset F$ so that
\begin{equation}\label{eq:R_covers}
X=\bigcup_{i=0}^{q}T^{i}R
\end{equation}
\begin{equation}\label{eq:F_separates}
F\cap T^{i}(F)=\emptyset,\,i=1,2,\ldots,m-1
\end{equation}

Let $\omega:X\rightarrow[0,1]$ be a continuous function so that $\omega_{|R}\equiv1$
and $\omega_{|F{}^{c}}\equiv0$. We define a random walk for $z\in X$.
At any point $p$ we arrive during the random walk, the walk terminates
with probability $\omega(p)$ and moves to $T^{-1}p$ with probability
$1-\omega(p)$. Notice that for every point $z\in X$
the random walk will terminate in at most $q$ steps. Indeed
by Equation (\ref{eq:R_covers}) there is an $i\in\{0,\ldots,q\}$ so that
$T^{-i}z\in R$ at which $1-\omega(T^{-i}z)=0$.  Conclude there is a finite number
of possible walks starting at $z$ and we denote by $f(z)$ the expected
length of the walk starting at $z$. As there is a uniform bound on
the number of walks, $f:X\rightarrow\mathbb{R}_{+}$
is continuous. Note that if $y\notin F$, then $f(y)=f(T^{-1}y)+1$. Thus if $f(Tx)\neq f(x)+1$ then $Tx\in F$. Conclude $E_{f}=\{x\in X\,|\, f(Tx)\neq f(x)+1\}\subset T^{-1}F$. By Equation (\ref{eq:F_separates}) $T^{i}(E_{f})\cap E_{f}=\emptyset$ for $i=1,\ldots,m-1$.

Now assume $(X,T)$ has the strong topological
Rokhlin property. Fix $n\in\mathbb{N}$ and let $f:X\rightarrow\mathbb{R}$
be a continuous function such that for the open set $E_{f}=\{x\in X\,|\, f(Tx)\neq f(x)+1\}$,
$T^{i}(E_{f})$, $i=0,1,\ldots,n-1$, are pairwise disjoint. We claim
that the iterates $E_{f}$,$T^{-1}E_{f},\ldots$ eventually cover
$X$ and therefore $E_{f}$ is an open $n$-marker. Indeed as $f$ is bounded from above for any $x\in X$, the
series $f(T^{i}x)$, $i=1,2,\ldots$ cannot increase indefinitely.
\end{proof}

\section{Finite dimensional Jaworski-type theorem\label{sec:Finite-dimensional-Jaworski-type}}
\begin{thm}
Let $d\in\mathbb{N}$. If $X$ is finite dimensional and $\perdim(X,T)<\frac{d}{2}$,
then $(X,T)$ can be equivariantly embedded in $(([0,1]^{d})^{\mathbb{Z}},\shift)$.\end{thm}
\begin{proof}
Let $\dim(X)=n\in\mathbb{N}$. We will use the following representation
$(X\times X)\setminus(\triangle\cup(P\times P))\subset D_{1}\cup D_{2}$
where $D_{1}=(X\times X)\setminus(\triangle\cup(P_{6n}\times X)\cup(X\times P_{6n}))$
and $D_{2}=(P_{6n}\times(X\setminus P))\cup((X\setminus P)\times P_{6n})$. That is $D_{1}$ is the collection of pairs $(x,y)\in X\times X$,  where $x,y$ are distinct elements whose orbits have size greater than $6\dim(X)$; $D_{2}$ is the collection of pairs $(x,y)\in X\times X$, where $x$ has orbit of size less or equal than $6\dim(X)$ and $y$ is aperiodic or vice versa. By Propositions \ref{prop:D_U times V is dense (non periodic case)} and \ref{prop:D U times V is dense (mixed periodic case)}, for $(x,y)\in D_{1}\cup D_{2}$ one may find open neighborhoods $U_1$ and $U_2$ of $x$ and $y$ respectively so that for $K_{(x,y)}=\overline{U}_{1}\times\overline{U}_{2}$,
$D_{K_{(x,y)}}$ is dense in $C(X,[0,1]^{d})$ and $K_{(x,y)}\cap\triangle=\emptyset$. By Lemma \ref{lem:D_K is open} for
any $(x,y)\in D_{1}\cup D_{2}$, $D_{K_{(x,y)}}$
is open in $C(X,[0,1]^{d})$. As $X\times X$ is second-countable,
every subspace is a Lindelöf space, i.e.
every open cover has a countable subcover. We can therefore choose an open
countable cover of $D_{1}\cup D_{2}=(X\times X)\setminus(\triangle\cup(P\times P))$
by $\mathcal{U}=\{\mathring{K}_{(x_{1},y_{1})},\mathring{K}_{(x_{2},y_{2})},\ldots\}$. Thus there is a closed countable cover of $D_{1}\cup D_{2}$
by closed sets $\overline{\mathcal{U}}\triangleq\{K_{(x_{1},y_{1})},K_{(x_{2},y_{2})},\ldots\}$
so that $D_{K_{(x_{i},y_{i})}}$ is open and dense in $C(X,[0,1]^{n})$
for $i\in\mathbb{N}$. By Theorem \ref{thm:periodic case of embedding conjecture} there
is a comeagre set $\mathcal{D}\subset C(X,[0,1]^{d})$ so that if
$f\in\mathcal{D}$, then $I_{f}$ is $\big ((P\times P)\setminus\triangle\big )$-compatible.
Define the comeagre set $\mathcal{A}=\mathcal{D}\cap\bigcap_{K\in\mathcal{\overline{\mathcal{U}}}}D_{K}$.
Clearly any $f\in\mathcal{A}$ realizes an embedding $(X,T)\hookrightarrow(([0,1]^{d})^{\mathbb{Z}},\shift)$
(see Subsection \ref{sub:The-Baire-Category} for more explanations).
\end{proof}

\begin{prop}
\label{prop:D_U times V is dense (non periodic case)}
Given $x,y\in X$ with orbits of size greater than $6dimX$, there exist open neighborhoods $U_1$ and $U_2$ of $x$ and $y$ respectively such that for $K_{(x,y)}=\overline{U}_{1}\times\overline{U}_{2}$, the set of $K_{(x,y)}$-compatible functions
is dense in $C(X,[0,1]^{d})$ and $K_{(x,y)}\cap\triangle=\emptyset$.
\end{prop}
\begin{proof}
First we claim we can find $i_{0}=0,i_{1},\ldots,i_{2n}$
so that $T^{i_{0}}x,\ldots,T^{i_{2n}}x,T^{i_{0}}y,\ldots,T^{i_{2n}}y$
are pairwise distinct. Indeed if $y\notin \orb(x)$ this is trivial.
Likewise if $x\notin P$ and $y=T^{l}x$ for some $l\in\mathbb{N}$,
define $i_{k}=k(l+1)$, $k=0,1,\ldots,2n$. Finally if $y\in \orb(x)$
and $x\in P$, then by assumption $x\in P_{N}\setminus P_{N-1}$ for
$N>6n$. To find $i_{0}=0,i_{1},\ldots,i_{2n}$ invoke Lemma \ref{lem:one third good indices}.
We can thus find open neighborhoods $U_1$ and $U_2$ of $x$ and $y$ respectively so that $\{T^{i_{k}}\overline{U}_{x},T^{i_{k}}\overline{V}_{y}\}_{k=0}^{2n}$
are pairwise disjoint. Denote $K_{(x,y)}=\overline{U}_{1}\times\overline{U}_{2}$.
Let $\epsilon>0$. Let $\tilde{f}:X\rightarrow[0,1]^{d}$ be a continuous
function. We will show that there exists a continuous function $f:X\rightarrow[0,1]^{d}$
so that $\|f-\tilde{f}\|_{\infty}<\epsilon$ and $I_{f}$ is $K_{(x,y)}$-compatible.
Let $\alpha_{1}$ and $\alpha_{2}$ be open covers of $\overline{U}_{1}$
and $\overline{U}_{2}$ respectively with $\max_{W\in\alpha_{j},k\in\{0,1,\ldots,2n\}}\diam(\tilde{f}(T^{i_{k}}W))<\frac{\epsilon}{2}$ and $ord(\alpha_{j})\leq n$ open neighborhoods $U_1$ and $U_2$ of $x$ and $y$ respectively
for $j=1,2$. For each $W\in\alpha_{j}$ choose $q_{W}\in W$ so that $\{q_{W}\}_{W\in \alpha_j}$  is a collection of distinct points in $X$.
Define $\tilde{v}_{W}=(\tilde{f}(T^{i_{k}}q_{W}))_{k=0}^{2n}$. By
Lemma \ref{lem:F periodic disjoint case} (with $n_{1}=n_{2}=2n+1$,
note $F_{2}^{\oplus n_1}=F_{2}$) one can find for $j=1,2$ continuous
functions $F_{j}:\overline{U}_{j}\rightarrow([0,1]^{d})^{2n+1}$, with the following
properties:
\begin{enumerate}
\item \label{enu:F is nearly v}$\forall W\in\alpha_{j}$, $||F(q_{W})-\tilde{v}_{W}||_{\infty}<\frac{\epsilon}{2}$,
\item \label{enu:convex approximation}$\forall x\in \overline{U}_{j}$, $F_{j}(x)\in \co\{F_{j}(q_{W})|\, x\in W\in\alpha_{j}\}$,
\item \label{enu:F(x)neqF(y)}If $x'\in\overline{U}_{1}$ and $y'\in\overline{U}_{2}$
then $F_{1}(x')\neq F_{2}(y')$.
\end{enumerate}

Let $A=\bigcup_{k=0}^{2n}T^{i_{k}}(\overline{U}_{1}\cup\overline{U}_{2})$.
Define $f':A\rightarrow[0,1]^{d}$:
\[
f'_{|T^{i_{k}}\overline{U}_{j}}(T^{i_{k}}z)=F_{j}(z)|_{k},\,k=0,1,\ldots,2n
\]
Fix $z\in\overline{U}_{j}$ and $k\in\{0,1,\ldots,2n\}$. As by property
(\ref{enu:convex approximation}), $f'(T^{i_{k}}z)=F_{j}(z)|_{k}\in \co\{F_{j}(q_{W})|_{k}|\, z\in W\in\alpha_{j}\}$,
$||f'(T^{i_{k}}z)-\tilde{f}(T^{i_{k}}z)||_{\infty}\leq\max_{z\in W\in\alpha_{j}}||F_{j}(q_{W})|_{k}-\tilde{f}(T^{i_{k}}z)||_{\infty}$.
Fix $W\in\alpha_{j}$ with $z\in W$. Note $||F_{j}(q_{W})|_{k}-\tilde{f}(T^{i_{k}}z)||_{\infty}\leq||F_{j}(q_{W})|_{k}-\tilde{v}_{W}|_{k}||_{\infty}+||\tilde{v}_{W}|_{k}-\tilde{f}(T^{i_{k}}z)||_{\infty}$.
The first term on the right-hand side is bounded by $\frac{\epsilon}{2}$
by property (\ref{enu:F is nearly v}). As $\diam(\tilde{f}(T^{i_{k}}W))<\frac{\epsilon}{2}$
and $\tilde{v}_{W}|_{k}=\tilde{f}(T^{i_{k}}q_{W})$ we have $||\tilde{f}(T^{i_{k}}q_{W})-\tilde{f}(T^{i_{k}}z)||_{\infty}<\frac{\epsilon}{2}$.
We finally conclude $||f'-\tilde{f||}_{A,\infty}<\epsilon$. By Lemma
\ref{lem:Tietze Extension} there is $f:X\rightarrow[0,1]^{d}$ so
that $f{}_{|A}=f'_{|A}$ and $||f-\tilde{f||}_{\infty}<\epsilon$.

Assume for a contradiction $I_{f}(x')=I_{f}(y')$ for some $(x',y')\in\overline{U}_{1}\times\overline{U}_{2}$.
In particular we have $(f(T^{i_{0}}x'),\ldots,f(T^{i_{2n}}x'))=F(x')=(f(T^{i_{0}}y'),\ldots,f(T^{i_{2n}}y'))=F(y')$
which is a contradiction to property (\ref{enu:F(x)neqF(y)}).
\end{proof}

\begin{prop}
\label{prop:D U times V is dense (mixed periodic case)}
Given $x,y\in X$ where $x$ has orbit of size less or equal than $6\dim(X)$ and $y$ is aperiodic or vice versa, there exist open neighborhoods $U_1$ and $U_2$ of $x$ and $y$ respectively such that for $K_{(x,y)}=\overline{U}_{1}\times\overline{U}_{2}$, the set of $K_{(x,y)}$-compatible functions
is dense in $C(X,[0,1]^{d})$ and $K_{(x,y)}\cap\triangle=\emptyset$.
\end{prop}

\begin{proof}

Assume w.l.o.g that $(x,y)\in(X\setminus P)\times H_{m}$ with $1\leq m\leq6n$.
Choose $S\in\mathbb{N}$ so that $Sd\geq \ord(\gamma)+1+md$. Let $U\subset X\setminus P_{6n}$
be an open set with $x\in U$ and so that $\overline{U}\cap T^{l}\overline{U}=\emptyset$,
$l=1,\ldots S-1$. As in Proposition \ref{Prop: U_x times U_y is dense} we choose $W$ to be an open set in $H_{m}$ with $y\in W\subset\overline{W}\subset H_{m}$. Define $K_{(x,y)}=\overline{U}\times\overline{W}$.
 Let $\tilde{f}:X\rightarrow[0,1]^{d}$ be a continuous function. We
will show that there exists a continuous function $f:X\rightarrow[0,1]^{d}$
so that $\|f-\tilde{f}\|_{\infty}<\epsilon$ and $I_{f}$ is $K_{(x,y)}$-compatible.
Let $\gamma$ be an open cover of $\overline{U}$ so that $\max_{V\in\gamma,k\in\{1,\ldots,S\}}\diam(\tilde{f}(T^{k}V))<\frac{\epsilon}{2}$ and $ord(\gamma)\leq n$. For
each $V\in\gamma$ choose $q_{V}\in V$ so that $\{q_{V}\}_{V\in \gamma}$ is a collection of distinct points in $\overline{U}$ and define $\tilde{v}_{V}=(\tilde{f}(T^{k}q_{V}))_{k=1}^{S}$.
By Lemma \ref{lem:F not in V_n}, there
exists a continuous function $F:\overline{U}\rightarrow([0,1]^{d})^{S}$, with
the following properties:
\begin{enumerate}
\item $\forall V\in\gamma$, $||F(q_{V})-\tilde{v}_{V}||_{\infty}<\frac{\epsilon}{2}$\label{enu:approximately v_U-1},
\item $\forall x\in \overline{U}$, $F(x)\in \co\{F(q_{V})|\, x\in V\in\gamma\}$\label{enu:convex combination-2},
\item \label{enu:not in V_n-1-1}For any $w\in \overline{U}$ it holds $F(w)\notin V_{S}^{m}$, where $V_{S}^{m}\triangleq\{z=(z_{1},\ldots,z_{S})\in([0,1]^{d})^{S}|\quad\forall 1\leq a,b\leq S,\ (a=b\mod m)\rightarrow z_{a}=z_{b}\}$.
\end{enumerate}
\noindent
Let $A=\bigcup_{k=1}^{S}T^{k}\overline{U}$. Define $f':A\rightarrow[0,1]^{d}$:
\[
f'_{|T^{k}\overline{U}}(T^{k}z)=F(z)|_{k},\,k=1,\ldots,S
\]
\noindent
We conclude as in Proposition \ref{prop:D_U times V is dense (non periodic case)}
that there is $f:X\rightarrow[0,1]^{d}$ so that $f_{|A}=f'_{|A}$
and $||f-\tilde{f}||_{\infty}<\epsilon$.

Let $(x',y')\in\overline{U}\times\overline{W}$. Note $(f(Tx'),\ldots,f(T^{S}x'))=F(x')\notin V_{S}^{m}$.
As $y'\in \overline{W}\subset H_{m}$, we have $(f(Ty'),\ldots,f(T^{S}y'))\in V_{S}^{m}$. Conclude by property (\ref{enu:not in V_n-1-1}), $I_{f}(x')\neq I_{f}(y')$.

\end{proof}

\section{Infinite dimensional Jaworski-type theorem}\label{sec:Infinite-dimensional-Jaworski-type}
\begin{thm}
\label{thm:16+1 Embedding}Let $d\in\mathbb{N}$. Let $(Z,S)$ be
an aperiodic finite-dimensional t.d.s and $(X,T)$ a t.d.s with $\mdim(X,T)<\frac{d}{16}$.
Let $\pi:(X,T)\rightarrow(Z,S)$ be an extension. Then there is a
continuous mapping $f:X\rightarrow[0,1]^{d+1}$ such that $I_{f}:(X,T)\hookrightarrow(([0,1]^{d+1})^{\mathbb{Z}},\shift)$
is an embedding.\end{thm}
\begin{proof}
By Theorem \ref{thm:I_f cross pi embedding} below, there is an embedding
$I_{g}\times\pi:(X,T)\hookrightarrow(([0,1]^{d})^{\mathbb{Z}},\shift)\times(Z,S)$
for some continuous mapping $g:X\rightarrow[0,1]^{d}$. By Jaworski's theorem (see Subsection \ref{sub:Jaworski's-Theorems}) there is a
continuous mapping $h:Z\rightarrow[0,1]$ so that $I_{h}:(Z,S)\hookrightarrow(([0,1])^{\mathbb{Z}},\shift)$
is an embedding. Conclude that for $f=(g,h\circ\pi):X\rightarrow[0,1]^{d+1}$,
$I_{f}:(X,T)\hookrightarrow(([0,1]^{d+1})^{\mathbb{Z}},\shift)$ is
an embedding.\end{proof}
\begin{defn}
Let $X,Y$ be metric spaces and $\epsilon>0$. A continuous mapping
$f:X\rightarrow Y$ such that for every $y\in Y$, $\diam(f^{-1}(y))<\epsilon$
is called an \textbf{$\epsilon$-embedding}.

The following theorem is closely related to \cite[Theorem 5.1]{L99}
and  \cite[Theorem 1.5]{GutTsu12}.\end{defn}
\begin{thm}
\label{thm:I_f cross pi embedding}Let $d\in\mathbb{N}$. Let $(Z,S)$
be an aperiodic finite-dimensional t.d.s and $(X,T)$ a t.d.s with
$\mdim(X,T)<\frac{d}{16}$. Let $\pi:(X,T)\rightarrow(Z,S)$ be an
extension. Then the collection of continuous mappings $f:X\rightarrow[0,1]^{d}$
that $I_{f}\times\pi:(X,T)\hookrightarrow(([0,1]^{d})^{\mathbb{Z}},\shift)\times(Z,S)$
is an embedding is comeagre in $C(X,[0,1]^{d})$.\end{thm}
\begin{proof}
Fix $\epsilon>0$. Denote:

\begin{equation}
D_{\epsilon}^{\pi}=\{f\in C(X,[0,1]^{d})|\quad I_{f}\times\pi\mathrm{\mathrm{\textrm{ is an }}\epsilon\textrm{-embedding}\}}\label{eq:D_epsilon^pi}
\end{equation}
By Lemma \ref{lem:D_epsilon is open} below $D_{\epsilon}^{\pi}$
is open in $C(X,[0,1]^{d})$. We now prove that $D_{\epsilon}^{\pi}$
is dense in $C(X,[0,1]^{d})$. As every $f\in\bigcap_{n=1}^{\infty}D_{\frac{1}{n}}^{\pi}$
realizes an embedding $I_{f}\times\pi:(X,T)\hookrightarrow(([0,1]^{d})^{\mathbb{Z}},\shift)\times(Z,S)$,
this implies the statement of the theorem through the Baire category theorem (see Subsection \ref{sub:The-Baire-Category}). Let $\tilde{f}:X\rightarrow[0,1]^{d}$
be a continuous function and $\delta>0$. We will show that there exists a continuous
function $f:X\rightarrow[0,1]^{d}$ so that $\|f-\tilde{f}\|_{\infty}<\delta$
and $I_{f}\times\pi$ is an $\epsilon$-embedding. We define:

\[
m_{\dim}=\begin{cases}
\frac{1}{32} & \quad \mdim(X,T)=0\\
\mdim(X,T) & \quad\mathrm{otherwise}
\end{cases}
\]
Notice it holds $m_{\dim}<\frac{d}{16}$. Let $\alpha$ be a cover
of $X$ with $\max_{U\in\alpha}\diam(\tilde{f}(U))<\frac{\delta}{2}$
and $\max_{U\in\alpha}\diam(U)<\epsilon$. Let $\epsilon'>0$ be such
that $16m_{\dim}(1+2\epsilon')<d$. Let $N\in\mathbb{N}$ be such that
it holds $\frac{1}{N}D(\alpha^{N+1})<(1+\epsilon')m_{\dim}$ (here
we use $m_{\dim}>0$ and Remark \ref{Rem: mdim with linear shift})
and $N$ is divisible by $16$. Let $\gamma\succ\alpha^{N+1}$ be
an open cover so that $D(\alpha^{N+1})=\ord(\gamma)$. We have thus
$\ord(\gamma)<N(1+\epsilon')m_{\dim}.$ Let $M=\frac{N}{2}$ and $S=\frac{M}{8}$.
Observe $M,S\in\mathbb{N}$. Notice $Sd>\frac{N}{16}16m_{\dim}(1+2\epsilon')=Nm_{\dim}(1+2\epsilon')>Nm_{\dim}(1+\epsilon')$.
Conclude:

\begin{equation}
\ord(\gamma)<Sd\label{eq:ord<Sd}
\end{equation}
Choose distinct $q_{U}\in U$, $U\in\gamma$ and define $\tilde{v}_{U}=(\tilde{f}(T^{i}q_{U}))_{i=0}^{N-1}$.
By Equation (\ref{eq:ord<Sd}), according to Lemma \ref{lem:Lin99 5.6} (\cite[Lemma 5.6]{L99}), one
can find a continuous function $F:X\rightarrow([0,1]^{d})^{N}$, with
the following properties:
\begin{enumerate}
\item $\forall U\in\gamma$, $||F(q_{U})-\tilde{v}_{U}||_{\infty}<\frac{\delta}{2}$,
\item $\forall x\in X$, $F(x)\in \co\{F(q_{U})|\, x\in U\in\gamma\}$,
\item \label{enu:linear-independence-1}If for some $0\leq l<N-4S$ and
$\lambda\in(0,1]$ and $x,y,x',y'\in X$ so that:
\end{enumerate}

\[
\lambda F(x)|_{l}^{l+4S-1}+(1-\lambda)F(y)|_{l+1}^{l+4S}=\lambda F(x')|_{l}^{l+4S-1}+(1-\lambda)F(y')|_{l+1}^{l+4S}
\]
then there exist $U\in\gamma$ so that $x,x'\in U$.

By Theorem \ref{thm:finite dimensional has local marker property}
$(Z,S)$ has the marker property. By Theorem \ref{thm:-Marker Property implies Strong Rokhlin}
$(Z,S)$ has the strong topological Rokhlin property. We therefore
have a continuous function $n:Z\rightarrow\mathbb{R}$ so that \textit{for}
$E_{n}=\{y\in Z\,|\, n(Sy)\neq n(y)+1\}$, it holds that $E_n\cap S^{i}(E_{n})$ for
$i=1,2,\ldots,2M-2$ are pairwise disjoint. Let $\underline{n}(x),\overline{n}(x)n'(x):X\rightarrow\mathbb{R}$ be given by
$\underline{n}(x)=\lfloor n(\pi(x))\rfloor\,\mod M$, $\overline{n}(x)=\lceil n(\pi(x))\rceil\,\mod M$,
$n'(x)=\{n(\pi(x))\}$. Define for $x\in X$:

\begin{equation}
f(x)=(1-n'(x)F(T^{-\underline{n}(x)}x)|_{\underline{n}(x)}+n'(x)F(T^{-\overline{n}(x)}x)|_{\overline{n}(x)}\label{eq:f'-1}
\end{equation}
 We claim $f$ is continuous. If $n(\pi(x))\notin \mathbb{Z}$, then $f$ is continuous at $x$ as $\underline{n}(x),\overline{n}(x)n'(x)$ are continuous at $x$.  For $n(\pi(x))\in \mathbb{Z}$ fix $\eta>0$ and a neighborhood $U$ of $x$ such that for all $x'\in U$, $|n(x)-n(x')|<\eta$ and $||F(T^{-n(x)}x)-F(T^{n(x)}x')||_{\infty}<\eta$.
Assume w.l.o.g $n(x)\leq n(x')<n(x)+\eta$ and compute, using $f(x)=F(T^{-n(x)}x)|_{n(x)}$, $$||f(x)-f(x')||_{\infty}\leq||n'(x')F(T^{-\overline{n}(x')}x')|_{\overline{n}(x)}||_{\infty}+||F(T^{-n(x)}x)-F(T^{n(x)}x')||_{\infty}<2\eta$$
We claim $||\tilde{f}-f||_{\infty}<\delta$. Indeed note that $f(x)\in \co\{F(q_{V})|_n|\,0\leq n<M,\, T^{-n}x\in V\in\gamma\}$, and thus it is enough to to show for $x\in X$ and $0\leq n<M$ such that  $T^{-n}x\in V\in\gamma$, that $||\tilde{f}(x)-F(q_{V})|_n||_{\infty}<\delta$. As $\gamma\succ\alpha^{N+1}$ and $M=\frac{N}{2}$, we may find $U\in \alpha$ such that $x, T^n q_{V}\in U$. As $\max_{U\in\alpha}\diam(\tilde{f}(U))<\frac{\delta}{2}$
and $\max_{U\in\gamma}||F(q_{U})-\tilde{v}_{U}||_{\infty}<\frac{\delta}{2}$, we conclude  $||\tilde{f}(x)-F(q_{V})|_n||_{\infty}\leq||\tilde{f}(x)-\tilde{f}(T^n q_{V})||_{\infty}+||F(q_{U})|_n-\tilde{v}_{U}|_n||_{\infty}<\delta$.

We now show that $f\in D_{\epsilon}^{\pi}$.
Fix $x,y\in X$. Assume for a contradiction $f(T^{a}x)=f(T^{a}y)$
for all $a\in\mathbb{Z}$ and $\pi(x)=\pi(y)$. The latter equality
implies $\pi(T^{a}x)=\pi(T^{a}y)$ for all $a\in\mathbb{Z}$, which
implies $n(\pi(T^{a}x))=n(\pi(T^{a}y))$ for all $a\in\mathbb{Z}$.
Notice that by Remark \ref{rem:at most one bad index (global case)}
there is at most one index $-\frac{3}{2}M\leq j\leq\frac{M}{2}-2$
for which $n(T^{j+1}x)\neq n(T^{j}x)+1$. By Lemma \ref{lem:existence of one good long segment}
one can find an index $-\frac{3}{2}M\leq r\leq0$ so that for $r\leq s\leq r+\frac{M}{2}-1$, $\underline{n}(T^{s}x)=\underline{n}(T^{r}x)+s-r$, $\overline{n}(T^{s}x)=\overline{n}(T^{r}x)+s-r$
and $\underline{n}(T^{r}x)\leq\frac{M}{2}$. Denote $\lambda=n'(T^{r}x)$,
$a=\underline{n}(T^{r}x)\leq\frac{M}{2}$ . Recall $\frac{M}{2}=4S$.
Substituting $T^{s}x,T^{s}y$ for $r\leq s\leq r+4S-1$ in Equation
(\ref{eq:f'-1}), we conclude from the equality $I_{f}(x)|_{r}^{r+4S-1}=I_{f}(y)|_{r}^{r+4S-1}$
and $n(\pi(T^{a}x))=n(\pi(T^{a}y))$ for all $a\in\mathbb{Z}$:

\[
(1-\lambda)F(T^{r-a}x)|_{a}^{a+4S-1}+\lambda F(T^{r-a-1}x)|_{a+1}^{a+4S}=(1-\lambda)F(T^{r-a}y)|_{a}^{a+4S-1}+\lambda F(T^{r-a-1}y)|_{a+1}^{a+4S}
\]
As $1-\lambda\neq 0$, by property (\ref{enu:linear-independence-1}) above, there exist $U\in\gamma\succ\alpha^{N+1}$
so that $T^{r-a}x,T^{r-a}y\in U$. As $N=-\frac{3}{2}M-\frac{M}{2}\leq r-a\leq0$
we can find $V\in\alpha$, so that $x,y\in V$. As $x,y$ were arbitrary
and $\max_{U\in\alpha}\diam(U)<\epsilon$, we conclude that for every
$z\in([0,1]^{d})^{\mathbb{Z}}\times Z$, $\diam((I{}_{f}\times\pi)^{-1}(z)<\epsilon$.

\end{proof}
\begin{lem}
\label{lem:D_epsilon is open}Let $\epsilon>0$. $D_{\epsilon}^{\pi}$
is open in $C(X,[0,1]^{d})$.\end{lem}
\begin{proof}
Let $f\in D_{\epsilon}^{\pi}$. We claim there exist $\tau>0$ so
that for every $z\in([0,1]^{d})^{\mathbb{Z}}\times F\triangleq R$,
$\diam((I{}_{f}\times\pi)^{-1}(B_{\tau}(z)))<\epsilon$. Notice that
by the definition of $D_{\epsilon}^{\pi}$ (Equation (\ref{eq:D_epsilon^pi})),
the previous claim is true when $B_{\tau}(z)$ is replaced by $z$. Now assume for a contradiction
that the claim does not hold. Then for all $i\in\mathbb{N},$ one
can find $z_{i}\in R$, $x_{i},y_{i}\in X$, , so that $I{}_{f}\times\pi(x_{i}),I{}_{f}\times\pi(y_{i})\in B_{\frac{1}{i}}(z_{i})$
and $d(x_{i},y_{i})\geq\epsilon$. By compactness there exist $x,y\in X$
so that $d(x,y)\geq\epsilon$ and $I{}_{f}\times\pi(x)=I{}_{f}\times\pi(y)$.
This is a contradiction to $f\in D_{\epsilon}^{\pi}$. Now let $g\in C(X,[0,1]^{d})$
so that $||f-g||_{\infty}<\frac{\tau}{2}$. We claim $g\in D_{\epsilon}^{\pi}$.
If $I{}_{g}\times\pi(x)=I{}_{g}\times\pi(y)$, then $||I{}_{f}\times\pi(x)-I{}_{f}\times\pi(y)||_{\infty}<\tau$,
which implies $d(x,y)<\epsilon$.
\end{proof}
\appendix
\renewcommand{\thesection}{\hspace*{-8pt}}
\section{}
\renewcommand{\thesection}{\Alph{section}}
\renewcommand{\thelem}{\Alph{lemma}}
\begin{lem}
\label{lem:closed marker iff open} Let $n\in\mathbb{N}$. $(X,T)$
has a closed $n$-marker iff $(X,T)$ has an open $n$-marker.\end{lem}
\begin{proof}
Let $F$ be a closed $n$-marker such that $\{T^{i}(F)\}_{i=1}^{m}$
cover $X$ for some $m\in\mathbb{N}$. Let $\epsilon>0$ be such that $B_{\epsilon}(F)\cap T^{i}B_{\epsilon}(F)=\emptyset$ for $i=1,2,\dots,n-1$. Define $U=B_{\epsilon}(F)$. Clearly $U$ is
an open $n$-marker. Now let $U$ be an open $n$-marker. Let $\epsilon>0$
be the Lebesgue number for the covering $\{T^{i}(U)\}_{i=1}^{m}$, i.e.
for every $x\in X$, there exists $1\leq i\leq m$ so that $B_{\epsilon}(x)\subset T^{i}(U)$.
Let $0<\delta<\epsilon$ be so that $B_{-\frac{\epsilon}{2}}(T^{i}U)\subset T^{i}\overline{B}_{-\delta}(U)$
for $0\leq i<m$. Note that for every $x\in X$, there exists $1\leq i\leq m$
so that $x\in B_{-\frac{\epsilon}{2}}(T^{i}U)$. Define $F=\overline{B}_{-\delta}(U)$.
Clearly $F$ is a closed $n$-marker.\end{proof}
\begin{lem}
\label{lem:D_K is open}For any $K\subset(X\times X)\setminus\triangle$
compact, $D_{K}$ is open in \textup{$C(X,[0,1]^{d})$}.\end{lem}
\begin{proof}
We repeat the argument of Lemma III.4 on p.21 of \cite{J74}. Fix
$f\in D_{K}$. Define $\alpha:K\rightarrow[0,1]$ given by $\alpha(x,y)\triangleq\sup_{n\in\mathbb{Z}}||f(T^{n}x)-f(T^{n}y)||_{\infty}$.
Notice that $\alpha_{|K}>0$ and that $\alpha$ is lower semicontinuous,
i.e. for any $a>0$ $\{(x,y)\in K|\,\alpha(x,y)\leq a\}$ is closed.
The latter implies $\alpha$ attains a minimum $\epsilon$ on $K$
and the former implies $\epsilon>0$. Conclude that for $g\in C(X,[0,1]^{d})$,
with $||g-f||_{\infty}<\frac{\epsilon}{3}$, one has $g\in D_{K}$
as $||g(T^{n}x)-g(T^{n}y)||_{\infty}\geq||f(T^{n}x)-f(T^{n}y)||_{\infty}-\frac{2\epsilon}{3}\geq\frac{\epsilon}{3}$
for any $(x,y)\in K$.
\end{proof}
Let $N\in\mathbb{N}$. Recall we denote $\mathbb{Z}_{N}=$$(\mathbb{Z}/N\mathbb{Z},\oplus)$.
\begin{lem}
\label{lem:one third good indices}Let $N\in\mathbb{N}$. Let $y\in\mathbb{Z}_{N}\setminus\{0\}$.
Then there exists $A\subset\mathbb{Z}_{N}$ with $0\in A$ so that
$|A|\geq\lceil\frac{N}{3}\rceil$and $(y\oplus A)\cap A=\emptyset$.\end{lem}
\begin{proof}
Start by defining $A={\{0\}}$. Note $(y\oplus A)\cap A=\emptyset.$
Let $|A'|<\lceil\frac{N}{3}\rceil$ with $(y\oplus A')\cap A'=\emptyset.$
Define $B=A'\cup(y\oplus A)\cup(-y\oplus A)$. Note $|B|<N$. Choose
$a\in\mathbb{Z}_{N}\setminus B$. Redefine $A=A'\cup\{a\}$ and note
$(y\oplus A)\cap A=\emptyset.$\end{proof}
\begin{lem}
\label{lem:existence of one good long segment}Let $n:X\rightarrow\mathbb{R}$
be a function, $x\in X$ and $M\in\mathbb{N},$ an even integer. Assume
there is at most one index $-\frac{3}{2}M\leq j\leq\frac{M}{2}-2$
for which \textup{$n(T^{j+1}x)\neq n(T^{j}x)+1$. Then one can find
an index $-\frac{3}{2}M\leq r\leq0$ so that $\lfloor n(T^{r}x)\rfloor\,\mod M\leq\frac{M}{2}$
and for $r\leq s\leq r+\frac{M}{2}-1$: }

\begin{equation}
\lceil n(T^{s}x)\rceil\,\mod M=\lceil n(T^{r}x)\rceil\,\mod M+s-r\label{eq:round up mod-1}
\end{equation}

\begin{equation}
\lfloor n(T^{s}x)\rfloor\,\mod M=\lfloor n(T^{r}x)\rfloor\,\mod M+s-r\label{eq:round down mod-1}
\end{equation}
\end{lem}
\begin{proof}
As the segment $-\frac{3}{2}M\leq j\leq\frac{M}{2}-2$ has total length
$M+(M-1)$, one of the intervals $\{-\frac{3}{2}M,-\frac{3}{2}M+1,\ldots,j\}$,$\{j+1,\ldots,\frac{M}{2}-2\}$
is at least of length $M$. Denote this interval by $\{a,\ldots b\}$
($b-a\geq M-1$). Note $n(T^{s}x)\,\mod M=n(T^{a}x)+s-a\, \mod M$
for $a\leq s\leq b$. Conclude there is $a\leq k\leq b$ so that $n(T^{k}x)\,\mod M\in[0,1)$.
If $b-k+1\geq\frac{M}{2}$, then let $r=k$. In this case $\lfloor n(T^{r}x)\rfloor\,\mod M=0$
and $n(T^{s}x_{i})\,\mod M\in[0,1+s-r)\subset[0,M-1)$ for $r\leq s\leq r+\frac{M}{2}-1$.
If $b-k+1<\frac{M}{2}$, then $k-a\geq b-a+2-\frac{M}{2}$ which implies
$\frac{M}{2}+1\leq k-a$ and we set $r=k-\frac{M}{2}-1\geq a$. In
this case $\lfloor n(T^{r}x)\rfloor\,\mod M\in[\frac{M}{2}-1,\frac{M}{2})$and
$n(T^{s}x_{i})\,\mod M\in[\frac{M}{2}-1,\frac{M}{2}+r-s)\subset[\frac{M}{2}-1,M-1)$
for $r\leq s\leq r+\frac{M}{2}-1$. In both cases we see that (\ref{eq:round up mod-1})
and (\ref{eq:round down mod-1}) hold. This also implies $r\leq-1$
as $r+\frac{M}{2}-1\leq\frac{M}{2}-2$. \end{proof}
\begin{lem}
\label{lem:Tietze Extension}Let $B\subset X$ be a closed set and
$\epsilon>0$. Let $f':B\rightarrow[0,1]^{d}$ and $\tilde{f}:X\rightarrow[0,1]^{d}$
be continuous functions so that $||\tilde{f}_{|B}-f'||_{B,\infty}<\epsilon$.
Then there exists a continuous function $f:X\rightarrow[0,1]^{d}$
so that $f_{|B}=f'$ and $||\tilde{f}-f||_{\infty}<\epsilon$.\end{lem}
\begin{proof}
Apply the Tietze extension theorem (\cite[Theorem 3.2]{M75}) to $\tilde{f}_{|B}-f':B\rightarrow[-\epsilon',\epsilon']^{d}$
for some $\epsilon>\epsilon'>0$, to find a continuous function $g:X\rightarrow[-\epsilon',\epsilon']^{d}$
so that $g_{|B}=f'-\tilde{f}_{|B}$, i.e. on $B$, $f'=\tilde{f}_{|B}+g$.
Let $h:\mathbb{R}\rightarrow[0,1]$ be given by $h(t)=1$ for $t\geq1$,
$h(t)=t$ for $0\leq t\leq1$ and $h(t)=0$ for $t\leq0$. Let $\vec{h}:\mathbb{R}^{d}\rightarrow[0,1]^{d}$ be
given by $\vec{h}(x_{1},\ldots,h_{d})=(h(x_{1}),\ldots,h(x_{d}))$.
Finally define:
\[
f=\vec{h}\circ(\tilde{f}+g)
\]
Clearly $f\in C(X,[0,1]^{d})$. Fix $x\in X$ and $0\leq i\leq d-1$.
If $\tilde{f}(x)|_{i}+g(x)|_{i}\in[0,1]$, then $f(x)|_{i}=\tilde{f}(x)|_{i}+g(x)|_{i}$
and $|f(x)|_{i}-\tilde{f}(x)|_{i}|\leq\epsilon'<\epsilon$. If $\tilde{f}(x)|_{i}+g(x)|_{i}\geq1$,
then $1+\epsilon'\geq\tilde{f}(x)|_{i}+g(x)|_{i}\geq1=f(x)|_{i}$
and we have $|f(x)|_{i}-\tilde{f}(x)|_{i}|\leq\epsilon'<\epsilon.$
Similarly for the case $\tilde{f}(x)|_{i}+g(x)|_{i}\leq0$. \end{proof}
\begin{lem}
\label{lem:Almost sure linear independence}Let $m,s\in\mathbb{N}$,
$r\in\mathbb{Z}_{+}$. Let $V\subset\mathbb{R}^{m}$ be a linear subspace
with $\dim(V)=r$. If $r+s\leq m$, then almost surely w.r.t Lebesgue
measure for $(v_{r+1},v_{r+2},\ldots,v_{r+s})\in([0,1]^{m})^{s}$, $\dim(V+\spann((v_{r+1},v_{r+2},\ldots,v_{r+s}))=r+s$,
equivalently $V\cap \spann(v_{r+1},v_{r+2},\ldots,v_{r+s})=\{\vec{0}\}$.
\end{lem}
\begin{proof}
By induction. Assume that almost surely in $(v_{r+1},\ldots v_{i})\in([0,1]^{m})^{i}$,
$i<m$, one has $\dim(V_{i})=i$ for $V_{i}=V+\spann(v_{r+1},v_{r+2},\ldots,v_{i})$.
For any $v_{i+1}\notin V_{i}\cap[0,1]^{m}$, one has $\dim(V+\spann(v_{r+1},v_{r+2},\ldots,v_{i+1}))=i+1$.
Note $\mu(V_{i}\cap[0,1]^{m})=0$ for $\mu$ the Lebesgue measure
in $[0,1]^{m}$. The result now follows from Fubini's theorem. \end{proof}
\begin{defn}
\label{Def: Affinely Independent}$v_{1},\ldots v_{l}\in\mathbb{R}^{m}$
are called \textbf{affinely independent} if for any $\lambda_{i}\in\mathbb{R}$,
$i=1,\ldots,l$ with $\Sigma_{i=1}^{l}\lambda_{i}v_{i}=0$, the equality
$\Sigma_{i=1}^{l}\lambda_{i}=0$ implies $\lambda_{i}=0$ for all
$i=1,\ldots,l$. This is equivalent to the requirement that $v_{2}-v_{1},v_{3}-v_{1},\ldots v_{l}-v_{1}$
are linearly independent. \end{defn}
\begin{lem}
\label{lem:Almost sure affine independence}Let $m,s\in\mathbb{N}$,
$r\in\mathbb{Z}_{+}$. Let $V=\spann(v_{1},\ldots,v_{r})\subset\mathbb{R}^{m}$
be a linear subspace with $\dim(V)=r$. If $r+s\leq m+1$, then almost
surely w.r.t Lebesgue measure for $(v_{r+1},v_{2},\ldots,v_{r+s})\in([0,1]^{m})^{s}$,
$v_{1},v_{2},\ldots,v_{r+s}$ are affinely independent. \end{lem}
\begin{proof}
By Lemma \ref{lem:Almost sure linear independence}, if $r+s\leq m$
then almost surely in $(v_{r+1},v_{2},\ldots,v_{r+s})\in([0,1]^{m})^{s}$,
$v_{1},v_{2},\ldots,v_{r+s}$ are linearly independent, a fortiori
affinely independent. We now consider the case $r+s=m+1$. Using the
previous case, almost surely for $(v_{r+1},v_{2},\ldots,v_{m})\in([0,1]^{m})^{m-r}$,
$v_{1},v_{2},\ldots,v_{m}$ are linearly independent. Let $V=\spann(v_{2}-v_{1},\ldots,v_{m}-v_{1})$.
Note that almost surely $\dim(V)=m-1$, which implies $\mu(V\cap[0,1]^{m})=0$
for $\mu$ the Lebesgue measure in $[0,1]^{m}$. Therefore almost
surely for $v'_{m+1}\in[0,1]^{m}$, $\dim(V+\spann(v'_{m+1}))=m$. Define
$v_{m+1}=v'_{m+1}+v_{1}$. Conclude that almost surely $\spann(v_{2}-v_{1},\ldots,v_{m}-v_{1},v_{m+1}-v_{1})=m$.
The result now follows from Fubini's theorem. \end{proof}
\begin{lem}
\label{lem:F periodic disjoint case}Let $\epsilon>0$. Let $n_{1},n_{2},d\in\mathbb{N}$
with $n_{1}\geq n_{2}$. Let $R_{1},R_{2}\subset X$ be closed sets
so that $R_{1}\cap R_{2}=\emptyset$. For $i=1,2$, let $\alpha_{i}$
be open covers of $R_{i}$ with $\ord(\alpha_{i})<\frac{n_{i}d}{2}$.
Assume that for each $i=1,2$ a collection of distinct points $\{q_{U}\}_{U\in \alpha_i}\subset R_i$  and  a collection of vectors $\{\tilde{v}_{U}\}_{U\in \alpha_i}\subset([0,1]^{d})^{n_{i}}$ are given, then
there exist continuous functions $F_{i}:R_{i}\rightarrow([0,1]^{d})^{n_{i}}$,
with the following properties ($i=1,2$):
\begin{enumerate}
\item \label{enu:Very Close}$\forall U\in\alpha_{i}$, $||F_{i}(q_{U})-\tilde{v}_{U}||_{\infty}<\frac{\epsilon}{2}$,
\item \label{enu:convex combination-1-2}$\forall x\in R_i$, $F_{i}(x)\in \co\{F_i(q_{U})|\, x\in U\in\alpha_{i}\}$,
\item \label{enu:F(x)neqF oplus n (y)}If $x'\in R_{1}$ and $y'\in R_{2}$
then $F_{1}(x')\neq(F_{2}(y'))^{\oplus n_{1}}$.
\end{enumerate}
\end{lem}
\begin{proof}
Note that as $\alpha_{i}$ are open covers of disjoint sets, then
if $U\in\alpha_{1}$ and $V\in\alpha_{2}$, then $U\cap V=\emptyset.$
Let $\{\psi_{U}\}_{U\in\alpha_{i}}$ be a partition of unity subordinate
to $\alpha_{i}$ so that $\psi_{U}(q_U)=1$. Let $\vec{v}_{U}\in([0,1]^{d})^{n_{1}}$, $U\in\alpha_{1}$
be vectors that will be specified later. Let $\vec{v}_{U}=\tilde{v}_{U}$,
$U\in\alpha_{2}$. Define $F_{i}:R_{i}\rightarrow([0,1]^{d})^{n_{i}}$
by:

\[
F_{i}(x)=\sum_{U\in\alpha_{i}}\psi_{U}(x)\vec{v}_{U}
\]

\noindent Thus property (\ref{enu:convex combination-1-2}) holds. For $x\in R_{i}$ define $\alpha_{i,x}=\{U\in\alpha_{i}|\,\psi_{U}(x)>0\}$.
Write (\ref{enu:F(x)neqF oplus n (y)}) explicitly as:

\begin{equation}
\sum_{U\in\alpha_{1,x'}}\psi_{U}(x')\vec{v}_{U}\neq\sum_{U\in\alpha_{2,y'}}\psi_{U}(y')(\vec{v}_{U})^{\oplus n_{1}}\label{explicit equality-1}
\end{equation}
The left-hand side (right-hand side) of Equation (\ref{explicit equality-1})
is a convex combination of at most $\ord(\alpha_{1})+1$ ($\ord(\alpha_{2})+1$)
vectors. Therefore the total number of vectors is bounded by $|\alpha_{1,x'}\cup\alpha_{2,y'}|\leq\frac{n_{1}d}{2}-\frac{1}{2}+1+\frac{n_{2}d}{2}-\frac{1}{2}+1\leq n_{1}d+1$.
Denote $V_{2}=\spann((\vec{v}_{U})^{\oplus n_{1}}|\, U\in\alpha_{2,y'})$
and note $\dim(V_{2})\leq|\alpha_{2,y'}|\leq\frac{n_{2}d}{2}+\frac{1}{2}$.
We distinguish between two cases. If $\dim(V_{2})=\frac{n_{2}d}{2}+\frac{1}{2}$
then $\{(\vec{v}_{U})^{\oplus n_{1}}|\, U\in\alpha_{2,y'}\}$ are linearly
independent and we can invoke Lemma \ref{lem:Almost sure affine independence},
to conclude that almost surely in $(\vec{v}_{U})_{U\in\alpha_{1,x'}}\in([0,1]^{dn_{1}})^{|\alpha_{1,x'}|}$,
$\{\vec{v}_{U}\}_{U\in\alpha_{1,x'}}\cup\{(\vec{v}_{U})^{\oplus n_{1}}\}_{U\in\alpha_{2,y'}}$
are affinely independent. This implies Equation (\ref{explicit equality-1})
holds almost surely as the sum of coefficients $\sum_{U\in\alpha_{1,x'}}\psi_{U}(x')-\sum_{U\in\alpha_{2,y'}}\psi_{U}(y')=0$
but all coefficients are positive (see Definition \ref{Def: Affinely Independent}).
If $\dim(V_{2})<\frac{n_{2}d}{2}+\frac{1}{2}$, then by Lemma \ref{lem:Almost sure linear independence}
almost surely in $(\vec{v}_{U})_{U\in\alpha_{1,x'}}\in([0,1]^{dn_{1}})^{|\alpha_{1,x'}|}$,
$\spann(\vec{v}_{U})_{U\in\alpha_{1,x'}}$ and $\spann((\vec{v}_{U})^{\oplus n_{1}}|\, U\in\alpha_{2,y'})$
are complementary linear subspaces and $\{\vec{v}_{U}\}_{U\in\alpha_{1,x'}}$
are linearly independent. This implies (\ref{explicit equality-1})
holds almost surely. Indeed notice that the null vector $\vec{0}$
is the only vector which belongs to both subspaces. As $\{\vec{v}_{U}\}_{U\in\alpha_{1,x'}}$
are linearly independent and $\psi_{U}(x')>0$ for all $U\in\alpha_{1,x'}$
the linear combination on the left hand side does not equal the null
vector. As there is a finite number of constraints of the form (\ref{explicit equality-1}),
we can therefore choose $\vec{v}_{U}\in([0,1]^{d})^{n_{1}}$, $U\in\alpha_{1}$
so that both property (\ref{enu:Very Close}) and property (\ref{enu:F(x)neqF oplus n (y)})
hold. Finally  we have already remarked property (\ref{enu:convex combination-1-2}) holds.
\end{proof}
\begin{lem}
\label{lem:F periodic translation case }Let $\epsilon>0$. Let $n,l,d\in\mathbb{N}$
with $1\leq l\leq n-1$. Let $R\subset X$ be a closed set. Let $\alpha$ be
an open cover of $R$ with $\ord(\alpha)<\frac{nd}{2}$. Assume
a collection of distinct points $\{q_{U}\}_{U\in \alpha}\subset R$  and  a collection of vectors $\{\tilde{v}_{U}\}_{U\in \alpha}\subset([0,1]^{d})^{n_{i}}$ are given, then
there exists a continuous function $F:R\rightarrow([0,1]^{d})^{n}$,
with the following properties:\end{lem}
\begin{enumerate}
\item \label{enu:Very Close-2}$\forall U\in\alpha$, $||F(q_{U})-\tilde{v}_{U}||_{\infty}<\frac{\epsilon}{2}$,
\item \label{enu:convex combination-1-1-2}$\forall x\in X$, $F(x)\in \co\{F(q_{U})|\, x\in U\in\alpha\}$,
\item \label{enu:F(x)neq bullet F(y)}If $x',y'\in R$ then $F(x')\neq(F(y'))^{\bullet l}$.\end{enumerate}
\begin{proof}
Let $\{\psi_{U}\}_{U\in\alpha}$ be a partition of unity subordinate
to $\alpha$ so that $\psi_{U}(q_U)=1$. Let $\vec{v}_{U}\in([0,1]^{d})^{n}$, $U\in\alpha$
be vectors that will be specified later. Define:
$$
F(x)=\sum_{U\in\alpha}\psi_{U}(x)\vec{v}_{U}
$$
\noindent Thus property (\ref{enu:convex combination-1-1-2}) holds. For $x\in R$ define $\alpha_{x}=\{U\in\alpha|\,\psi_{U}(x)>0\}$.
Considering $\vec{v}_{U}$ as vectors in $[0,1]^{dn},$ write (\ref{enu:F(x)neq bullet F(y)})
explicitly as ($\oplus$ denotes addition modulo $nd$, $S=\{0,1,\ldots nd-1\}$):

\begin{equation}
\sum_{U\in\alpha_{x'}}\psi_{U}(x')\vec{v}_{U}|_{i}-\sum_{U\in\alpha_{y'}}\psi_{U}(y')\vec{v}_{V}|_{i\oplus dl}\neq0,\, i\in S\label{explicit equality-1-1}
\end{equation}
\noindent
Consider this inequality as a symbolic equality in the (vector) symbols
$\vec{v}_{U}$, i.e., let $\alpha_{x'}=\{U_{1}^{1},\ldots,U_{m_{1}}^{1}\}$
and $\alpha_{y'}=\{U_{1}^{2},\ldots,U_{m_{2}}^{2}\}$ with $m_{1},m_{2}\leq \ord(\alpha)+1$,
and represent $\vec{v}_{U_{i}^{j}}$ as a vector composed of $nd$
symbols: $\vec{v}_{U_{i}^{j}}=(v_{U_{i}^{j}}^{k,g})_{k\in\{1,\ldots,d\},\, g\in\{1,\ldots,n\}}$.
Let $\mathcal{M}$ be the matrix constituting of all different (symbolic)
vectors appearing in Equation (\ref{explicit equality-1-1}).

\[
\mathcal{M}=\Big[\vec{v}_{U_{1}^{1}}{}_{|S},\cdots,\vec{v}_{U_{m_{1}}^{1}}{}_{|S},\vec{v}_{U_{1}^{2}}{}_{|S\oplus dl},\cdots,\vec{v}_{U_{m_{2}}^{2}}{}_{|S\oplus dl}\Big]
\]
The matrix has $nd$ rows. The number of columns of the matrix is
bounded by $2(\ord(\alpha)+1)\leq2(\frac{nd}{2}-\frac{1}{2}+1)=nd+1$. Different $x'$ and $y'$ may give rise to different matrices, however the total number of matrices that arise in this way is finite. To proceed in the proof it is useful that all matrices have the same dimensions. We thus introduce a finite number of additional symbolic vectors $\{w_{i}\}_{i=1}^l$, consisting of distinct symbols of those in $\{v_{U}\}_{U\in \alpha}$, and add them as columns to (some of) the matrices so that the resulting matrices, to which we refer as  \textit{augmented}, have exactly $nd$ rows and $nd+1$ columns. We then claim that almost surely in $\{v_{U}\}_{U\in \alpha}\cup\{w_{i}\}_{i=1}^l$ each of the augmented matrices consist of a affinely independent columns. By Fubini's theorem almost surely in $\{v_{U}\}_{U\in \alpha}$, each of the original matrices consist of  affinely independent columns. Thus Equation (\ref{explicit equality-1-1}) hold for all $x',y'\in R$.
This implies we may choose specific substitution of values into the symbolic vectors $\{v_{U}\}_{U\in \alpha}$  such  that
  properties (\ref{enu:Very Close}) and (\ref{enu:F(x)neqF oplus n (y)}) (as well as (\ref{enu:convex combination-1-1-2}) as remarked before)
hold for the resulting $F:R\rightarrow([0,1]^{d})^{n}$. All that is left to prove is the following statement: Given $k$ symbolic
column vectors of length $k-1$ in a $(k-1)\times k$ matrix $\mathcal{M}_{k}$
so that the same symbol does not appear twice in any row and in any
column and such that the same symbol appears at most twice in $\mathcal{M}_{k}$,
then these $k$ vectors are almost surely affinely independent.  The proof is
by induction on $k$.  The case $k=2$
is trivial as $\mathcal{M}$ is the $1\times2$ matrix $[a,b]$ and
almost surely in $(a,b)$, $a-b\neq0$. Assume the result holds for
$(k-1)\geq 2$. Let $\mathcal{M}_{k}=\Big[\vec{v}_{1},\cdots,\vec{v}_{k}\Big]$
where $\vec{v}_{j}=(v_{i,j})_{i=1}^{k-1}$. There are two cases.
In the first case every symbol in $\mathcal{M}_{k}$ appears exactly
once in $\mathcal{M}_{k}$. By Lemma \ref{lem:Almost sure affine independence}
(with $V=\emptyset$) almost surely $\mathcal{M}_{k}$ consists of
affinely independent vectors. In the second case there is some $v_{s,t}$
which appears exactly twice in $\mathcal{M}_{k}$, $v_{s,t}=v_{s',t'}$
for $s\neq s'$ and $t\neq t'$. Let $\mathcal{N}_{k}$ be the $k\times k$ matrix achieved from $\mathcal{M}_{k}$ by adding as a last row the vector $(1,1,\ldots,1)$. It is enough to show $\det(\mathcal{N}_{k})\neq 0$ almost surely. Denote by $P=\{v_{i,j}\}{}_{(i,j)\neq(s,t),(s',t')}$ the collection
all symbolic elements except $v_{s,t}=v_{s',t'}$. We can write the determinant of $\mathcal{N}_{k}$ in the following way:
\begin{equation}
\det(\mathcal{N}_{k})=\pm D_{(s,t),(s',t'),}(P)v_{s,t}^{2}+f_{1}(P)v_{s,t}+f_{0}(P)\label{eq:quadratic equation}
\end{equation}
where $D_{(s,t),(s',t'),}(P)$ is the determinant of the $(k-2)\times(k-2)$
minor $\mathcal{N}_{k-2}$ obtained from $\mathcal{N}_{k}$ by erasing
rows $s$ and $s'$ and columns $t$ and $t'$ and $f_{1},f_{0}$
are some functions all depending on $P$. Note that if $k>3$, $\mathcal{N}_{k-2}$ corresponds to the case $k-2$ of the claim and moreover the symbol $v_{s,t}$ does not appear in $\mathcal{N}_{k-2}$. If $k=3$, $\mathcal{N}_{k-2}=\mathcal{N}_{1}=[1]$. We conclude almost surely in $P$, $D_{(s,t),(s',t')}\neq 0$ and therefore  almost surely in $P\cup \{v_{s,t}\}$, $\det(\mathcal{N}_{k})\neq 0$.
\end{proof}

The following lemma is related to Lemma 6.5 of \cite{L99}.

\begin{lem} \label{lem:F not in V_n} Let  $m,d,S\in\mathbb{N}$ and $\epsilon>0$. Let $R\subset X$ be a closed set and $\gamma$ an open cover of $R$ with $Sd\geq \ord(\gamma)+1+md$.
Assume
a collection of distinct points $\{q_{U}\}_{U\in \gamma}\subset X$  and  a collection of vectors $\{\tilde{v}_{U}\}_{U\in \gamma}\subset([0,1]^{d})^{S}$ are given,
then there exists a continuous function $F:R\rightarrow([0,1]^{d})^{S}$,
with the following properties:
\begin{enumerate}
\item $\forall U\in\gamma$, $||F(q_{U})-\tilde{v}_{U}||_{\infty}<\frac{\epsilon}{2}$\label{enu:approximately v_U},
\item $\forall x\in R$, $F(x)\in \co\{F(q_{U})|\, x\in U\in\gamma\}$\label{enu:convex combination},
\item \label{enu:not in V_n-1}For any $x\in R$ it holds:

\begin{equation}
F(x)\notin V_{S}^{m}\label{eq:in lemma: not in V_n}
\end{equation}
\end{enumerate}
 where $V_{S}^{m}\triangleq\{y=(y_{1},\ldots,y_{S})\in([0,1]^{d})^{S}|\quad\forall 1\leq a,b\leq S,\ (a=b\mod m)\rightarrow y_{a}=y_{b}\}$.
 \end{lem}

\begin{proof} Let $\{\psi_{U}\}_{U\in\gamma}$ be a partition of
unity subordinate to $\gamma$ so that $\psi_{U}(q_{U})=1$. Let $\vec{v}_{U}\in([0,1]^{d})^{S}$,
$U\in\gamma$ be vectors that will be specified later. Define:

\[
F(x)=\sum_{U\in\gamma}\psi_{U}(x)\vec{v}_{U}
\]
\noindent Thus property (\ref{enu:convex combination}) holds. For $x\in R$ define $\gamma_{x}=\{U\in\gamma|\,\psi_{U}(x)>0\}$.
 Write (\ref{eq:in lemma: not in V_n})
explicitly as:

\begin{equation}
\sum_{U\in\gamma_{x}}\lambda_{j}\psi_{U}(x)\vec{v}_{U}
\notin V_{S}^{m}\label{explicit equality}
\end{equation}

Note that $\dim(V_{S}^{m})=md$ and $md+|\gamma_{x}|\leq md+\ord(\gamma)+1\leq Sd$.
By Lemma \ref{lem:Almost sure linear independence}, almost surely
in $\{\vec{v}_{U}\}_{U\in \gamma_x}$ in $\big([0,1]^{d})^{S}\big)^{|\gamma_{x}|}$,
$V_{S}^{m}\cap \spann(\{\vec{v}_{U}\}_{U\in \gamma_x})=\{\vec{0}\}$.  As there is a finite
number of constraints of the form (\ref{explicit equality}), we can
 choose $\vec{v}_{U}\in([0,1]^{d})^{S}$, $U\in\gamma$ so
that properties (\ref{enu:approximately v_U}) and (\ref{enu:not in V_n-1}) hold
(as well as (\ref{enu:convex combination-1}) as remarked before).
\end{proof}

The following lemma is Lemma 5.6 of \cite{L99}. The proof is similar to the proofs of Lemmas \ref{lem:F periodic disjoint case}, \ref{lem:F periodic translation case } and \ref{lem:F not in V_n}.
\begin{lem} \label{lem:Lin99 5.6} Let $\epsilon>0$ and $N,d,S\in\mathbb{N}$
with $N>4S$.  Let $\gamma$ be an open cover of $X$ with $\ord(\gamma)<Sd$.
Assume
a collection of distinct points $\{q_{U}\}_{U\in \gamma}\subset X$  and  a collection of vectors $\{\tilde{v}_{U}\}_{U\in \gamma}\subset([0,1]^{d})^{N}$ are given,
then there exists a continuous function $F:X\rightarrow([0,1]^{d})^{N}$,
with the following properties:
\begin{enumerate}
\item $\forall U\in\gamma$, $||F(q_{U})-\tilde{v}_{U}||_{\infty}<\frac{\epsilon}{2}$,
\item $\forall x\in X$, $F(x)\in \co\{F(q_{U})|\, x\in U\in\gamma\}$,
\item If for some $0\leq l<N-4S$ and
$\lambda\in(0,1]$ and $x,y,x',y'\in X$ so that:
\end{enumerate}
\[
\lambda F(x)|_{l}^{l+4S-1}+(1-\lambda)F(y)|_{l+1}^{l+4S}=\lambda F(x')|_{l}^{l+4S-1}+(1-\lambda)F(y')|_{l+1}^{l+4S}
\]
then there exist $U\in\gamma$ so that $x,x'\in U$.

\end{lem}

\bibliographystyle{alpha}
\bibliography{universal_bib}

\vspace{0.5cm}

\address{Yonatan Gutman, Institute of Mathematics, Polish Academy of Sciences,
ul. \'{S}niadeckich~8, 00-656 Warszawa, Poland.}

\textit{E-mail address}: \texttt{y.gutman@impan.pl}
\end{document}